\documentclass[10pt]{amsart}
\usepackage{amscd}
\usepackage[arrow,matrix]{xy}
\usepackage{graphicx}
\usepackage{amsmath}
\usepackage{amsmath, latexsym, amssymb}
\input xypic
\numberwithin{equation}{section}
\theoremstyle{plain}
\newtheorem{lemma}{Lemma}[subsection]
\newtheorem{proposition}[lemma]{Proposition}
\newtheorem{theorem}[lemma]{Theorem}

\theoremstyle{definition}
\newtheorem{definition}[lemma]{Definition}
\newtheorem{remark}[lemma]{Remark}
\newtheorem{example}[lemma]{Example}

\begin{document}
\newcommand{\R}{{\mathbb R}}
\newcommand{\C}{{\mathbb C}}
\newcommand{\F}{{\mathbb F}}
\renewcommand{\O}{{\mathbb O}}
\newcommand{\Z}{{\mathbb Z}} 
\newcommand{\N}{{\mathbb N}}
\newcommand{\Q}{{\mathbb Q}}
\renewcommand{\H}{{\mathbb H}}

\newcommand{\Aa}{{\mathcal A}}
\newcommand{\Bb}{{\mathcal B}}
\newcommand{\Cc}{{\mathcal C}}    %configuration space
\newcommand{\Dd}{{\mathcal D}}
\newcommand{\Ee}{{\mathcal E}}
\newcommand{\Ff}{{\mathcal F}}
\newcommand{\Gg}{{\mathcal G}}    %gauge transformations
\newcommand{\Hh}{{\mathcal H}}
\newcommand{\Kk}{{\mathcal K}}
\newcommand{\Ii}{{\mathcal I}}
\newcommand{\Jj}{{\mathcal J}}
\newcommand{\Ll}{{\mathcal L}}    %Loop space
\newcommand{\Mm}{{\mathcal M}}    %moduli space
\newcommand{\Nn}{{\mathcal N}}
\newcommand{\Oo}{{\mathcal O}}
\newcommand{\Pp}{{\mathcal P}}
\newcommand{\Qq}{{\mathcal Q}}
\newcommand{\Rr}{{\mathcal R}}
\newcommand{\Ss}{{\mathcal S}}
\newcommand{\Tt}{{\mathcal T}}
\newcommand{\Uu}{{\mathcal U}}
\newcommand{\Vv}{{\mathcal V}}
\newcommand{\Ww}{{\mathcal W}}
\newcommand{\Xx}{{\mathcal X}}
\newcommand{\Yy}{{\mathcal Y}}
\newcommand{\Zz}{{\mathcal Z}}

\newcommand{\zt}{{\tilde z}}
\newcommand{\xt}{{\tilde x}}
\newcommand{\Ht}{\widetilde{H}}
\newcommand{\ut}{{\tilde u}}
\newcommand{\Mt}{{\widetilde M}}
\newcommand{\Llt}{{\widetilde{\mathcal L}}}
\newcommand{\yt}{{\tilde y}}
\newcommand{\vt}{{\tilde v}}
\newcommand{\Ppt}{{\widetilde{\mathcal P}}}
\newcommand{\bp }{{\bar \partial}} 

\newcommand{\Remark}{{\it Remark}}
\newcommand{\Proof}{{\it Proof}}
\newcommand{\ad}{{\rm ad}}
\newcommand{\Om}{{\Omega}}
\newcommand{\om}{{\omega}}
\newcommand{\eps}{{\varepsilon}}
\newcommand{\Di}{{\rm Diff}}
\newcommand{\Pro}[1]{\noindent {\bf Proposition #1}}
\newcommand{\Thm}[1]{\noindent {\bf Theorem #1}}
\newcommand{\Lem}[1]{\noindent {\bf Lemma #1 }}
\newcommand{\An}[1]{\noindent {\bf Anmerkung #1}}
\newcommand{\Kor}[1]{\noindent {\bf Korollar #1}}
\newcommand{\Satz}[1]{\noindent {\bf Satz #1}}

\renewcommand{\a}{{\mathfrak a}}
\renewcommand{\b}{{\mathfrak b}}
\newcommand{\e}{{\mathfrak e}}
\renewcommand{\k}{{\mathfrak k}}
\newcommand{\pg}{{\mathfrak p}}
\newcommand{\g}{{\mathfrak g}}
\newcommand{\gl}{{\mathfrak gl}}
\newcommand{\h}{{\mathfrak h}}
\renewcommand{\l}{{\mathfrak l}}
\newcommand{\sm}{{\mathfrak m}}
\newcommand{\n}{{\mathfrak n}}
\newcommand{\s}{{\mathfrak s}}
\renewcommand{\o}{{\mathfrak o}}
\newcommand{\so}{{\mathfrak so}}
\renewcommand{\u}{{\mathfrak u}}
\newcommand{\su}{{\mathfrak su}}
\newcommand{\ssl}{{\mathfrak sl}}
\newcommand{\ssp}{{\mathfrak sp}}
\renewcommand{\t}{{\mathfrak t }}
\newcommand{\Cinf}{C^{\infty}}
\newcommand{\la}{\langle}
\newcommand{\ra}{\rangle}
\newcommand{\half}{\scriptstyle\frac{1}{2}}
\newcommand{\p}{{\partial}}
\newcommand{\notsub}{\not\subset}
\newcommand{\iI}{{I}}               %unit interval [0,1]
\newcommand{\bI}{{\partial I}}      %boundary of same
\newcommand{\LRA}{\Longrightarrow}
\newcommand{\LLA}{\Longleftarrow}
\newcommand{\lra}{\longrightarrow}
\newcommand{\LLR}{\Longleftrightarrow}
\newcommand{\lla}{\longleftarrow}
\newcommand{\INTO}{\hookrightarrow}

\newcommand{\QED}{\hfill$\Box$\medskip}
\newcommand{\UuU}{\Upsilon _{\delta}(H_0) \times \Uu _{\delta} (J_0)}
\newcommand{\bm}{\boldmath}

\title[Homogeneous spaces with invariant $G_2$-structures]{\large Classification of
compact homogeneous spaces  with invariant $G_2$-structures}
\author{H\^ong V\^an L\^e  and Mobeen Munir  } 
\thanks{H.V.L. is
partially supported by grant of ASCR Nr IAA100190701 and M.M.
is partially supported by  HEC of Pakistan}
\date{}

\maketitle

\medskip

\abstract In this note we classify all homogeneous spaces $G/H$ admitting a $G$-invariant $G_2$-structure, assuming that  $G$ is a connected compact Lie group and $G$ acts effectively on $G/H$. They include a subclass   of all homogeneous spaces $G/H$ with a $G$-invariant $\tilde G_2$-structure, where $G$ is a compact  Lie group. There are many  new  examples  with  nontrivial fundamental group. We study a subclass of homogeneous spaces  of high rigidity  and low rigidity and  show that  they  admit  families of invariant coclosed $G_2$-structures (resp. $\tilde G_2$-structures). %We   use  them to construct  new metrics of $Spin (3,4)$-holonomy.}
\endabstract
\tableofcontents


\section{ Introduction}

In recent years  manifolds admitting a $G_2$-structure  have attracted  increasing  interests of physicists and mathematicians.
These manifolds  can be  geometric models in the  theory of superstrings with torsion \cite{GMW2004}. In   another  field,  
a recent work of Donaldson and Segal \cite{DS2009} suggests that  a right framework for  a gauge theory in dimension 7  is a class of  manifolds  with non-vanishing torsion $G_2$-structure.
A main source of computable models of manifolds with $G_2$-structures are   homogeneous spaces  or spaces of co-homogeneity one \cite{II2005}, \cite{CF2006}, \cite{CS2002}.   

In this note  we classify all  compact homogeneous spaces  $M ^7$ of the form
$G/H$ such that $G$ is a  connected compact  Lie group  acting effectively on $G/H$, admitting a $G$-invariant  structure of $G_2$-type or
of the non-compact  form $ \tilde G_2$-type.  This classification extends the classification
by Friedrich-Kath-Moroianu-Semmelmann of all simply-connected compact homogeneous nearly parallel $G_2$-manifolds
in \cite{FKMS1997}. We   study    manifolds with $\tilde G_2$-structure, not only because
of their striking similarity with those  admitting a $G_2$-structure, but  they present an interesting class  in pseudo Riemannian
geometry \cite{Kath1998}. We also like to point out  that even the classification of symmmetric spaces  with holonomy contained in $\tilde G_2$ is open.

Recall that a 7-dimensional  smooth  manifold $M ^7$  is said to admit {\it a $G_2$-structure} (resp. {\it a $\tilde G_2$-structure}), if there is a section   of the
bundle $\Ff (M^7 )/G_2 $ (resp. $\Ff (M ^7)/\tilde G_2)$ over $M ^7$, where $\Ff (M ^7)$ is the frame bundle over $M^7$. 
It is well-known that
$G_2$ (resp. $\tilde G_2$)  is the automorphism group of a 3-form $\phi$ (resp. $\tilde \phi$)  on $\R ^7$, \cite{Reichel1907}, \cite[p. 114]{HL1982}, or \cite[p. 539]{Bryant1987}. Such  a 3-form $\phi$ (resp. $\tilde \phi$) is called  a {\it 3-form of $G_2$-type} (resp. {\it $\tilde G_2$-type}).
It is known that  the $GL(\R ^7)$-orbits of $\phi$ and
$\tilde \phi$ are the only open  orbits  of the $GL(\R ^7)$-action on $\Lambda  ^3  (\R ^7 ) ^*$, see e.g. 
\cite{Bryant1987}, \cite{Hitchin2000}, \cite{LPV2008}. Any 3-form on these open orbits is called {\it a stable 3-form}, \cite{Hitchin2000}, or
{\it a definite 3-form}, if it lies in the  orbit of $\phi$,  or an   {\it indefinite  3-form}, if it lies in the  orbit of $\tilde \phi$.
The existence of a $G_2$-structure  (resp. $\tilde G_2$-structure) on  a  manifold $M ^7$ is equivalent to the existence
of a definite  differential 3-form  $\phi$ (resp. indefinite  differential 3-form $\tilde \phi$)  on $M ^7$. 

The plan of our note is as follows.  In section 2 we collect important properties of the groups $\tilde G_2$ and $G_2$, which are needed for our classification.
In section 3 we classify   homogeneous manifolds   $G/H$ admitting invariant $\tilde G_2$-structures, where $G$ is a connected compact Lie group
and $H$ is a  closed Lie subgroup (not necessary connected) of $G$, see Theorem \ref{Theorem.2.3.1}.  This problem is equivalent to finding  all pairs $(G, H)$ where $H$ is a  closed (hence compact) subgroup of a compact Lie group $G$  such that the image of the isotropy representation
$\rho(H)$ is a   subgroup of $\tilde G_2\subset  Gl (7, \R)$. We observe that any  such homogeneous space $G/H$  also admits an invariant
$G_2$-structure,  hence  $\rho(H)$   is also a subgroup of $G_2 \subset Gl(7, \R)$.  
In section  4 we classify  all  homogeneous manifolds $G/H$ admitting invariant $G_2$-structures,  where $G$ is a compact Lie group
and $H$ is a  closed Lie subgroup (not necessary connected) of $G$, see Theorem \ref{Theorem.3.3.1}. Our classification is reduced to finding  all pairs $(G, H)$  such that the image of the isotropy representation
$\rho(H)$ is a   subgroup of $ G_2\subset  Gl (7, \R)$.  We also compute the dimension of the space of all $G$-invariant $G_2$-structures on a 
homogeneous manifold $G/H$, see Remark \ref{Remark.3.3.2}.a.
In section 5 we study a special class of homogeneous manifolds $G/H$ admitting invariant $G_2$-structures  using our classification.
Among these spaces there are many  known examples   of  manifolds   admitting  $G_2$-structures.  We  explain some known properties of these examples 
 using  simpler arguments based on our  classification. We also  present some  new results  concerning these spaces.   
 
Let us describe the method of our classification.   First we notice that $G/H$  admits a $G$-invariant $G_2$-structure (resp. $\tilde G_2$-structure), if and only if
it admits a $G$-invariant definite 3-form (resp.  indefinite  3-form). In the first step we find  all pairs  of   corresponding Lie algebras $(\h\subset \g)$. In the second step we  find  the  associated
pairs of Lie groups  $(H\subset G)$. The first step   is done  using  representation theory and is fairly standard, even it  could be done  using
some  special  software package. There is no algorithm known to solve the second  problem. So we have developed  a set of techniques to  find  the normalizer of a given connected Lie subgroup, and   after that we can find all  Lie  subgroups (not necessary connected) with a given Lie algebra  obtained in the first step.

Finally we remark that  the problem we solve in this note is a part of a more general question to  classify all homogeneous spaces $M=G/H$ admitting  $G$-invariant $\tilde G_2$-structures or  $G$-invariant $G_2$-structures.
If we require $M$ to be  compact  and with finite fundamental group,  by  the Montgomery theorem
\cite[Corollary 3]{Montgomery1950},  $M$  has also a transitive  action of a compact  subgroup $G' \subset G$. Thus  $G$ is a subgroup of  the  full  diffeomorphism group  of $M= G'/(G'\cap H)$  preserving a given  $G'$-invariant $\tilde G_2$- (resp. $G_2$-) structure on $M$. 

\section{ The groups $\tilde G_2$ and $G_2$}

In this section we  recall the definitions of  $\tilde G_2$ and $G_2$. We describe the  maximal compact subgroup
of $\tilde G_2$,  which is unique up to conjugacy by elements of $\tilde G_2$. We also describe maximal compact subgroups of $G_2$. These subgroups  are needed  for    our classifications in  sections 3 and 4.

\subsection{The group $\tilde G_2$ and its maximal compact subgroup $SO(4)$}

We refer the reader to \cite{Bryant1987}  for  a  definition and properties  of
the exceptional Lie group $\tilde G_2$.
For the convenience of the reader we  briefly
describe the group  $\tilde G_2$, which is less familiar  than its dual compact group $G_2$.

Let  us fix a  basis $e^1, \cdots , e^7$ in $(\R ^7) ^*$.
Denote by $\om ^{ijk}$ the 3-form $e ^i \wedge e ^j \wedge e ^k\in \Lambda ^3 (\R ^7) ^*$.

\medskip

\begin{definition}\label{Definition.2.1.1} \cite{Reichel1907}, see also  \cite[Definition 2, p.543]{Bryant1987}.  The group
$\tilde G_2$  is   defined as  the subgroup $\{ g \in GL(\R ^7)|\,  g ^* (\tilde \phi )  = \tilde \phi\}$ where
$$\tilde \phi = \om ^{123}-\om ^{145} -\om ^{167} -\om ^{246}+\om ^{257} +\om ^{347} + \om ^{356}.$$
\end{definition}
\medskip

\begin{lemma}\label{Lemma.2.1.3}\cite[Theorem 2]{Bryant1987} The group $\tilde G_2$ is the automorphism group of  the split-octonion algebra.  The group $\tilde G_2$ is connected.
\end{lemma}

Theorem 2 in  \cite{Bryant1987} cited above is given without a proof (but it can be proved in the same way as in the proof of \cite[Theorem 1]{Bryant1987}). 
A similar  explanation for the first assertion of Lemma \ref{Lemma.2.1.3} can be found in \cite[\S 6.2]{LPV2008}, where we  proved that $\tilde G_2$ is a subgroup of the automorphism  group of the Malcev simple algebra of dimension 7, which
is the imaginary part  $ Im\, \O_S$ of the  split-octonion algebra $\O_S$.  Since  the  multiplication on the Malcev algebra is the imaginary part of  the octonion multiplication  on  $ Im\, \O_S$, we get  easily  $\tilde G_2 \subset Aut (\O_S)$. The  other inclusion $Aut ( \O_S)\subset  \tilde G_2$   can be   verified straightforwardly.
A detailed proof  for the second assertion of Lemma \ref{Lemma.2.1.3} can be found in \cite{Le2006b} (the first version, which is also available at the arxiv server), namely this assertion is a direct consequence of  Lemmas 2.1 and 2.2 proved therein.

As a topological space, $\tilde G_2$ is a direct product of  its maximal compact Lie subgroup and
a vector space.

\begin{lemma}\label{Lemma.2.1.6}  The maximal compact subgroup of $\tilde G_2$ is $SO (4)$. The inclusion of $SO(4) \to \tilde G_2 \to Gl(\R ^ 7)$ acts on $\R^7$ with two irreducible subspaces of dimension 3 and dimension  4. Any compact subgroup of $\tilde G_2$ is conjugate to a
subgroup in the maximal compact subgroup $SO(4)$. 
\end{lemma}

 The first assertion of Lemma  \ref{Lemma.2.1.6} is known  to experts in the Cartan theory of real semisimple Lie groups but we don't find an explicit proof of it in standard text-books. In \cite[Corollary 2.4]{Le2006b} we  give
a  topological proof of  this assertion.
  For the convenience of the reader we  give here another algebraic proof,   which  explains also the second assertion of Lemma \ref{Lemma.2.1.6}. By \cite[Theorem 1.1, p.252]{Helgason1978} the maximal compact Lie subgroup of $\tilde G_2$ is connected whose Lie algebra is  a maximal compact Lie subalgebra in $\tilde \g_2$.
     Note that $\so(4)=\su(2)+\su(2)$ is a maximal compact Lie subalgebra
of $\tilde \g_2$  which can be described  in terms of the root decomposition  of the complex Lie algebra $\g _2 ^ \C$, namely it is  the intersection of the normal form $\tilde \g_2$  of $\g_2 ^\C$
and the compatible compact form $\g_2$. Using  the weights of the  representation
of the subalgebra $\su(2) +\su(2) \subset \tilde \g_2$ on $ \R ^7$,  it is easy to see that
the  corresponding connected Lie subgroup  in $\tilde G_2$ is $SO(4)$ and the corresponding representation is a sum
of two real irreducible  representations  of dimension  3 and dimension 4.
This proves the first and the second   assertion of  Lemma \ref{Lemma.2.1.6}.  The last assertion of Lemma \ref{Lemma.2.1.6} is a consequence of \cite[Theorem 2.1, p. 256]{Helgason1978}.

We now describe another way to construct    
an  explicit  embedding of $SO(4)$ into $\tilde G_2$, see \cite[chapter IV,(1.9), p. 115]{HL1982}, since it will be useful in our computations later. The group $Sp(1) \times Sp(1)$ acts on the split-octonion algebra $\O_S = \H \oplus \H e$ as follows:
\begin{equation}
\chi (q_1, q_2) ( a+ b e) : = (q_1 a \bar q_1+  q_2 b\bar q_1 e).
\label{2.1.7}
\end{equation}
It is easy to see  that 
this action defines an embedding  of $SO(4)$ into $\tilde G_2$. Thus  we  can regard  this maximal compact  subgroup $SO(4)$  as  the intersection  $\tilde G_2 \cap (SO (Im\, \H ) \times SO ( \H e) )$. Taking into account \cite[Theorem 1.1, p.252]{Helgason1978} this construction also gives
a proof of the first and the second assertion of Lemma \ref{Lemma.2.1.6}.

To distinguish an abstract Lie  group $SO(4)$ (resp. a Lie algebra $\so (4)$) with  its image inside $\tilde G_2$ (resp. $\tilde \g_2)$ we denote the later one by $SO(4)_{3,4}$ (resp. $\so(4)_{3,4}$). 
 Note that the conjugacy  class of
$SO(4)_{3,4}$ in $Gl(\R ^ 7)$ is defined uniquely by the highest weights of its representation.  
We denote by $\su(2)_{3,4}$ the Lie  subalgebra in  $\so(4)_{3,4}$ corresponding to  the Lie   subgroup $\{ \chi (q_1, 1)| q_1 \in Sp(1)\}$  in formula (\ref{2.1.7}),
and by $\su(2) _{0,4}$ the Lie subalgebra  corresponding to the Lie subgroup $\{ \chi(1, q_2)|, q_2\in Sp(1)\}$. The conjugacy   of $\so(4)_{3,4}= \su(2)_{3,4}+ \su(2)_{2,4}$ in  $gl(\R^7)$ is defined  uniquely up to conjugacy  by the highest weights  $(2,0)$ and $(1,1)$ of the  irreducible  components of the representation of  $\so(4)$  explained in   Lemma \ref{Lemma.2.1.6}  and in (\ref{2.1.7}). The weight $(2,0)$  corresponds to the  irreducible real representation of dimension 3, and the weight
(1,1) corresponds to the irreducible real representation of dimension 4. We refer the reader to  \cite[\S 8]{Onishchik2004}  for a comprehensive  exposition of the theory of real  representations of real semisimple
Lie algebras, or \cite[Appendix]{VO1988} for a compact exposition of the theory.
Since $SO(4)$ is connected, the  conjugacy class of an embedding $SO(4) \to Gl(\R^ 7)$ is defined uniquely by the
 representation  of its Lie algebra $\su(2)_1+ \su(2)_2$, where $\su(2)_1$  (resp. $\su(2)_2$)  is  the Lie algebra of the first  (resp.  the second) subgroup $Sp(1)$  defined just before (\ref{2.1.7}).  
 
We also remark that  there are  three non-conjugate subalgebras in $\so(4)_{3,4}$ which  are isomorphic to $\so(3) = \su(2)$. We denote by $\so(3)_{3,3}$ the third  Lie subalgebra in this  subclass.  It is  defined by the
diagonal embedding of $\so (3) =\su (2)$ into $\so (4)_{3,4}= \su (2)_{3,4} + \su (2)_{0,4}$.  
\medskip

We summarize  a part of  our discussion in the following

\begin{lemma}
\label{Lemma.2.1.8}   The image of a representation $\bar \chi: SO(4) \to Gl ( \R ^7)$ is conjugate  to $SO(4)_{3,4} \subset Gl(\R^7)$, if and only  $\bar \chi$ is  a  sum of two irreducible real representations, one
of dimension 3 with the  highest weight  $(2,0)$, and  the other  of dimension 4 with the highest weight $(1,1)$.
\end{lemma}

\subsection{The group $G_2$ and its maximal compact  subgroups}

\begin{definition}
\label{Definition.3.1.1} \cite{Reichel1907}, see also  \cite[IV.1.A, p.114]{HL1982}, and  \cite[Definition 1, p.539]{Bryant1987}.  The group
$G_2$  is   defined as  the subgroup $\{ g \in GL(\R ^7)|\,  g ^* (\tilde \phi )  = \tilde \phi\}$ where
$$ \phi = \om ^{123}+\om ^{145} +\om ^{167} +\om ^{246}-\om ^{257} -\om ^{347} - \om ^{356}.$$
\end{definition}

We observe that $\phi + \tilde \phi  = 2 \om ^{123}$.

Dynkin's classical result \cite{Dynkin1952} asserts that the Lie algebra $\g_2$ has exactly three (up to conjugation) maximal subalgebras of dimensions
8, 6 and 3 respectively: $\su (3), \so (4)_{3,4}, \so (3)_7$, from which  we have  seen $\so(4)_{3,4}$ in  the  previous subsection.
The Lie subalgebra $\su(3)$ is the intersection $\g_2 \cap \gl (\R^6) \subset \gl(\R ^7)$ for any embedding $\gl (\R ^6) \subset \gl(\R ^7)$, see e.g. \cite[\S 2]{CS2002} for a proof.
The Lie subalgebra $\so(3)_7$ is defined by a real irreducible representation of $\su(2)$ of real dimension 7.

Let us fix the  basis $(e_i)$  of $\R ^7$  dual to  the basis $(e ^i)$.
Denote by $D_7$ the element  $diag (-1, 1, -1, 1, -1, 1, -1)\in Gl(\R^7)$. It is easy to check that $D_7$ preserves the form $\phi$, hence
$D_7 \in G_2$.

For any  element $a$ of order $k$ in a group $G$ we denote by $\Z_k[a]$ the  cyclic subgroup in $G$ generated by  $a$.

\begin{lemma}\label{Lemma.3.1.3}  Any maximal  proper subgroup in $G_2$ is conjugate to one of the  following subgroups in $G_2$:
$SU(3) \cdot \Z_2[D_7]$, $SO(4)_{3,4}$, $SO(3)_7$. 
\end{lemma}

This Lemma  is likely known to experts (see e.g. \cite[\S 8, p.112]{CS2002} for a statement without a proof that the normalizer $\Nn_{G_2} (SU(3))$ is $SU(3)\cdot \Z_2[D_7]$), but we do not have a reference with a  proof of it.   
  For the convenience of the reader we  give here a proof of  Lemma \ref{Lemma.3.1.3}  using  the Dynkin result above,  combining with the   invariance principle as well as with  the Schur's Lemma and its consequence stated below.

- {\it Invariance principle}. Suppose that $H^0$ is a (connected) subgroup of $G\subset SO (W)$. We denote by  $U $ the fixed-point subspace   of the action of $H^0$ on $W$. Then the normalizer $\Nn_G(H^0)$ preserves  the  subspace $U$ and  its orthogonal complement $U^\perp$.

- {\it Schur's Lemma and its consequence}. Suppose that the inclusion $H^0 \to G \to Gl (\R ^n) \to Gl (\C ^n)$ gives a  complex irreducible
representation of $H^0$  in $Gl (\C ^n)$. Then the centralizer $\Zz_G (H^0)$ is equal to the center $\Zz (G)$ of $G$. Using this we can  compute
$\Nn_G (H ^ 0)$ easily, taking into account the relation $ Int (H^0)  \subset \Nn _G (H^0)/\Zz_G(H^0)\subset Aut(H^0)$.

  Applying the invariance  principle to  $W= \R^7, U = \R$,  we   conclude that  if $x\in\Nn_{G_2} (SU(3))$ then either
 $x\in  Gl (\R^6)\cap G_2 = SU(3) $ or $x\cdot D_7\in Gl(\R^6) \cap G_2 = SU(3)$, what  proves the  first assertion.
 To compute $\Nn_{G_2} ( SO(4)_{3,4})$  we  apply the invariance principle to the space
 $W =(\Lambda ^3 ( \R ^7)^* ) ^\perp_\phi$ which is  the orthogonal complement to $\langle \phi \rangle _\R$ in
 $\Lambda  ^3 (\R ^7 ) ^*$.  Note that $\phi _0 = \phi - 7 \om ^{123}$ is an element of $W$, and $U = \la  \phi_0\ra _\R$ is the  fixed point  subspace of the induced $SO(4)_{3,4}$-action on $W$. By the invariance principle   $U$ is   invariant under the induced action of $\Nn_{G_2} (SO(4)_{3,4})$. Note that  for $g\in \Nn_{G_2} (SO(4)_{3,4})$, we have $g^* (\phi _0) = \pm \phi_0$, since $g\in SO(7)$. If  $g ^* (\phi _0) = \phi_0$, then
 $g$ must belong to $SO(4)_{3,4}$. If  not, then $g ^* ( 7\om ^{123}) =  2 \phi - 7\om ^{123}$.
Taking into account that $g$ preserves the induced norm on $W\subset \Lambda  ^3 (\R ^7 ) ^*$, we   obtain  a  contradiction. Hence $\Nn_{G_2} (SO (4)_{3,4})= SO(4)_{3,4}$.  
 Using  the Schur's Lemma and its consequence, taking into account that $\Zz(G_2) = \Z_1$ \cite[p.516]{Helgason1978},
we conclude that  the normalizer $\Nn_{G_2} (SO(3) _7)$  is $SO(3) _7$, the connected Lie subgroup having Lie algebra $\so(3)_7$.

\section{ Compact homogeneous  manifolds admitting  invariant $\tilde G_2$-structures}

In this section  we  classify homogeneous manifolds   $G/H$ admitting $G$-invariant $\tilde G_2$-structures, where $G$ is a compact Lie group
and $H$ is a  closed Lie subgroup (not necessary connected) of $G$.  Since $H$ is a compact Lie group, this problem is equivalent to the classification  of all pairs $(G, H)$  such that the image of the isotropy representation
$\rho(H)$ is a    compact subgroup of $\tilde G_2\subset  Gl (7, \R)$.  In subsection 3.1 we reduce the classification problem to a  representation problem, which is essentially linear when we classify only  the corresponding  Lie algebras $(\g, \h)$. The hardest part is to  find all  disconnected closed Lie subgroups $H$ whose isotropy representation maps $H$ into a  subgroup of  $\tilde G_2$. In subsection 3.2 we  summarize  our  classification in a table. We also compute the dimension of the  space
of $G$-invariant $\tilde G_2$-structures on  each  manifold $G/H$.

\subsection{Reduction to a representation problem} In this subsection we  first find  Lie algebras $(\h\subset \g)$ of
  compact Lie groups $(H\subset G)$  such that $(G/H)$ admits a $G$-invariant $\tilde G_2$-structure, and  then we find
  the  corresponding pairs $(H\subset G)$. Though the first step is a standard technique, we  describe all these
  algebras in detail, since  we use this description  in the second step.

Let  $G$ be a connected compact Lie group which acts transitively on  a  connected compact  smooth  manifold $M^7 = G/H$.
Without lost of generality we can assume that
$G$ acts   effectively  on $M$.

Let $\langle , \rangle _\g$ be a left and right invariant metric
on $G$.
Denote  by $\rho$ the isotropy representation  of $H$ on the  tangent space  $T_{eH} G/H = \R ^7$.
Let $\g$ (resp. $\h$) be the Lie algebra of $G$  (resp. $H$).
We write $\g = \h + V$, where $V $ is  the orthogonal complement to $\h$ w.r.t. $\langle , \rangle _\g$. 
Denote by $\bar \rho$ the induced isotropy action of  $\h$ on $ V$.
Since the action of $G$ is  almost effective, 
$ker \bar \rho = 0$.

Taking into account  Lemma \ref{Lemma.2.1.6} and our  discussion at the end of subsection 2.1 we get immediately

\begin{lemma}\label{red1} $G/H$ admits a $G$-invariant $\tilde G_2$-structure if and only if $\rho (H)$ lies in a  compact subgroup $SO (4)_{3,4} \subset Gl(V)$.
Consequently, the Lie subalgebra $\bar\rho(\h) \subset \so (4)_{3,4}$  is  one of the following  subalgebras  

1) $\bar\rho( \h) = \so (4)_{3,4}$; (we shall use $``= "$, ``be", ``coincide with", ``equal to"  for ``be conjugate to", if  no misunderstanding arises).

2) $\bar\rho(\h) =  \so(3)$ with three possible embeddings into $\so(4)_{3,4}$:\\
(2a)- $\bar \rho (\h) = so(3)_{3,3}$;\\
(2b)-   $\bar \rho (\h)= \su (2)_{3,4}$;\\
(2c) - $\bar \rho (\h) =\su (2)_{0,4}$.

3) $\bar\rho(\h) =  \so(3)+ \R$  with two  possible embeddings into $\so(4)_{3,4}$;\\
(3a) - the summand $\so(3) \subset \bar\rho(\h)$ coincides with  $\su (2)_{3,4}$,\\
(3b) -  the summand $\so(3) \subset \bar\rho(\h)$ coincides with  $\su (2)_{0,4}$.

4) $\bar\rho( \h) =  \R ^2 $.

5) $\bar\rho( \h) =  \R ^1 = \so (2) $ (there are infinitely many inequivalent embeddings of  $\so (2)$ into $\so (4)$).

6) $\bar\rho(\h) =  0 $.
\end{lemma}

Let us  explain our method to find all pairs $(H \subset G)$   satisfying  the conditions in our classification.

By Levy decomposition theorem we can represent  $G $ as a quotient $(G ^{sc} \times T^k)/Z $,   where $G^{sc}$ is a connected simply-connected semisimple  compact Lie group  and  $Z$ is a  finite  central subgroup  of $\hat G = G ^{sc} \times T^k$.  Denote by
$p$ the projection $\hat G \to G$. Note that  the action of $\hat G $ on $\hat  G/ p ^{-1} (H)$ is almost effective.
Moreover the image of the  isotropy action of  $p^{-1}( H)$ on $V$  coincides with the image of the isotropy action
of $H$ on $V$. Hence $\hat G/p ^{-1}(H)$  admits a $\hat G $-invariant $\tilde G_2$-structure, if $G/H$ does.
Next we observe that the effectiveness of the action of $G$ on $G/H$ is equivalent to the relation $\Zz (G) \cap H  = Id$, assuming
that the action of $G$ on $G/H$ is almost effective, i.e. $\ker \bar \rho  = 0$.
This is  equivalent to the relation $\Zz (\hat G) \cap  p ^{-1} (H) = Z$. Under the assumption that $\hat G$ acts on $\hat G/H'$ almost effectively,   we reduce  a classification  of all pairs
$H\subset G$ satisfying our conditions  to a  classification of all pairs $(H'\subset \hat G)$
such that $\hat G /H'$ admits a $\hat G$-invariant  $\tilde G_2$-structure.  To get the corresponding groups $H\subset G$ we set $G =  \hat G/ (\Zz(\hat G)\cap H')$, $H = H'/ ( \Zz(\hat G)\cap  H')$. 
We  solve this problem  in the following steps. In the first step, for each  possibility   among (1) -(6) above, we   find all pairs
$(\h \subset \g)$  of a compact Lie algebra $\h$ of co-dimension 7  in a compact Lie algebra $\g$ such that the adjoint
representation $\bar \rho (\h)$ on $V$ is the  given possibility, moreover $\ker \bar \rho = 0$. Then we find a connected Lie  subgroup $H^0\subset \hat G$ with the given Lie
algebra $\h \subset \g$. As we have mentioned above, this step is fairly standard.

In the second step  we   find  all Lie subgroups $H$  in $\hat G$ with Lie algebra $\h$ obtained in the first step. 
This subgroup lies in the normalizer $\Nn_{\hat G}(H^0)$. It is an extension of a finite subgroup $\Gamma$ in $\Nn_{\hat G} (H^0)/H^0$ by $H^0$. 
In our  note  we  compute the normalizer  of a  connected Lie subgroup $H^0$ in a compact Lie group $G$ by using ad hoc methods for each separate case.   The invariance principle as well as the Schur's Lemma and its consequence   are also used frequently in  our consideration.

In the third step we  verify  if the isotropy action
of this  subgroup $H$ on $V$  lifts to an embedding into the group $SO (4)_{3,4}\subset \tilde G_2$. 

In the final step we compute   $\Zz (\hat G)\cap H$, knowing $\Zz (\hat G) = \Zz ( G^{sc}) \times T^k$.  The center $\Zz(G^{sc})$ is known, see e.g. table 10 in \cite{VO1988}.

Now we proceed to   consider each  possibility listed  in Lemma \ref{red1}.
 
 {\it  Possibility 1 with $\bar\rho( \h) = \so (4)_{3,4}$}. Taking into account Lemma \ref{Lemma.2.1.8} we conclude that
 $\g$ must be semisimple. Since $\dim \g = 13$  and $\g \supset \so(4)$,  we conclude that  $ \g = \so (5) + \so (3)$.

\begin{proposition} \label{Proposition.2.2.1} Suppose that $\hat G/H$  admits a $\hat G$-invariant  $\tilde G_2$-structure such that their corresponding Lie algebras $(\h\subset \g)$ are  in possibility 1. Then  $\hat G= G^{sc} = Sp(2)\times Sp(1)$. The corresponding Lie subgroup $H$ is either  $Sp(1)_1\times  Sp (1)_2$, or the normalizer $Sp(1)_1 \times  Sp (1)_2\times \Z_2[\Zz(Sp(1))]$, described in the proof below. The kernel of the   $G^{sc}$-action  is  $\Z_2$, or $\Z_2 \times \Z_2[\Zz(Sp(1))]$ respectively.
\end{proposition}

\begin{proof}  In this case
the embedding $\Pi: \h=\so(4) = \so(3)_1 +\so(3)_2 \to \g= \so (5) +\so(3) $  is  defined as follows.
$\Pi$ is a direct sum of the canonical embedding $\Pi ^0 : \h = \ssp (1)_1 + \ssp(1)_2 \to \ssp(2) = \so (5) \subset \g$
and the projection  $\Pi ^1 $  from $\h$ to the ideal $\so(3)  \subset \g$.  In this note  we use  frequently isomorphism $\ssp (1) = \so (3)=\su(2)$, so  $\ssp (1)_i$ denotes the same subalgebra $\so (3)_i$, i =1,2.
The space $V$ is $W + W ^\perp$, where $W$ is the orthogonal complement of $\Pi^0 (\ssp (1)_1 + \ssp(1)_2)$
in $\ssp(2)$ and $W^\perp$ is  the orthogonal complement of $W$ in $V$. 
We also denote by  $\Pi$ the lift  of the representation
$\Pi$  to the corresponding simply connected Lie group $G^{sc}$.  Let $Sp(1)_i$ be the corresponding  Lie subgroup in $G^{sc} = Sp(2) \times Sp(1)$
with Lie subalgebra $\ssp(1)_i$.
Below we  decompose
$\ssp(2) = \Pi^0 (\h) +W$ in  a matrix expression, cf. \cite[p. 446]{Helgason1978}. 

$$
\ssp(2) = \left(\begin{array}{cccc}
ia_1 & w_1 & z_1& w_2\\
 - \bar w_1 & ia_2 & w_2& z_2\\
-\bar z_1 &- \bar w_2 & -ia_1 & w_1\\
- \bar w_2 & -\bar z_2&  -\bar w_1 & -ia_2\\
\end{array}\right) \subset \su(4),  w_i , z_i \in \C \text{ and } a_i \in \R.
$$

The subspace $W$  consists of those matrices with vanishing $a_i$ and $z_i$.
Here  is a matrix representation of 

$$
W ^\perp =\{  \left(\begin{array}{cccc}
0 & 0 & 0& 0\\
0 & -ia_2 & 0& -z_2\\
 0& 0 & 0 & 0\\
0 & \bar z_2&  0 & ia_2\\
\end{array}\right), \left  (\begin{array}{cc}
 ia_2 & z_2\\
 -\bar z_2&   -ia_2\\
\end{array}\right)\} \in \ssp (2) + \ssp(1).
$$
By Lemma \ref{Lemma.2.1.8}, the image of  the adjoint
representation $\rho(Sp(1)_1 \times Sp(1)_2)$  on $V = W+ W ^{\perp}$ is $SO(4)_{3,4}$. 
Using the invariance principle, we conclude that the normalizer of $Sp(1)_1 \times  Sp (1)_2$
in $G ^{sc} = Sp (2) \times Sp(1)$ is  $(Sp (1)_1 \times Sp (1)_2)\times \Zz(Sp(1))$.  This proves  the first and the second assertion of 
Proposition \ref{Proposition.2.2.1}. The last assertion  follows from a direct computation.
\end{proof}

{\it  Possibility 2   with $\h = \so (3)$}. Recall that there are three sub-cases (2a), (2b) and (2c). We denote by $SO(3)_{3,3}$  (resp.  $SU(2)_{3,4}$, $SU(2)_{2,4}$)
the    connected Lie subgroup in $SO(4)_{3,4}$ whose Lie algebra is $\so(3)_{3,3}$ (resp. $\su(2)_{3,4}$, $\su(2)_{2,4}$).

 From Lemma \ref{Lemma.2.1.8}  we get  immediately

\begin{lemma}
\label{Lemma.2.2.2}   An embedding $\Pi : so(3)\to   gl ( \R ^7)$ can be factored
as an embedding  $\Pi: so(3)  \to so(4)_{3,4} \subset gl ( \R ^7)$,  if and only if  one of the following 
three conditions holds.

Case (i). $\Pi$  is a direct sum of  two real irreducible representations of dimension 3 and  one trivial  representation.
In this case the  image of the induced embedding $\Pi _* (\so (3))$ is $\so(3)_{3,3}$ associated with  case (2a).

Case (ii). $\Pi$  is  a sum of one real irreducible representation of dimension 4  and one real irreducible representation  of dimension 3. In this case the image of the induced embedding $\Pi _* (\so (3))$ is $\su(2)_{3,4}$  associated with  case (2b).

Case (iii). $\Pi$ is  a sum of a real irreducible representation of dimension 4 and  three real representations of dimension 1. In this case the image of the induced embedding $\Pi _* (\so (3))$  is $\su(2)_{0,4}$ associated  with  case (2c). 
\end{lemma}

 Lemma \ref{Lemma.2.2.2} implies that $\g$ cannot contain a commutative ideal of dimension greater than or equal  to 4. Taking into account $\dim \g = 10$, we  conclude that  $\g$ must be one of the following
Lie algebras:\\ 
i) $\g = \so (5)$,\\
ii)  $\g = \su(3) + \R ^2$,\\
iii) $\g = 3 \so(3) + \R$.

Let us denote the element $diag (1, -1, -1, -1, -1)\in SO(5)$ by $D_{1,4}$.  We denote by $p$ the projection from $Spin (5)$ to $SO(5)$.
Then $p^{-1}(\Z_2 [ D_{1,4}]) = \Z_2 \times \Z_2 \subset Spin (5)$.

Let $H$  be  a Lie subgroup of   a Lie group $G$ and $\Gamma$ be a finite subgroup  of  the normalizer  $\Nn_G (H)$. We denote by $H \cdot \Gamma$ the
Lie  subgroup in $G$ generated by $H$ and $\Gamma$.   If the intersection $\Gamma$ with $H$ is   the neutral element $e \in G$,   and $\Gamma$ is a subgroup of  the  centralizer $\Zz_G(H)$, then we also write $H \times \Gamma$ instead of 
$H \cdot \Gamma$.

\begin{proposition}
\label{Proposition.2.2.3}  Suppose that $\hat G/H$  admits a $\hat G$-invariant  $\tilde G_2$-structure such that their corresponding Lie algebras $(\h\subset \g)$ are  in possibility 2. Then one of the following case happens.\\
 Case (i) with $\hat G = Spin (5)= Sp(2)$. Then $H$ is  conjugate to one of the  following  subgroups\\
-   $Sp(1)\cdot \Gamma$, where $Sp(1)$ is diagonally embedded  into  $Sp(1) \times Sp(1)  \subset  Sp(2)=Spin(5)$ (case (2a)) and $\Gamma \subset p ^{-1}(\Z_2 [ D_{1,4}])$. The kernel of the action is  $\Zz(Spin(5)) = \Z_2$.\\
- $ Sp(1) \times \Gamma$, where $Sp(1)$ is the canonically embedded $Sp (1) \subset Sp (2)$ (case (2c)) and $\Gamma$ is a finite  subgroup in  $ Sp (1)_2 \subset \Zz_{Sp(2)} (Sp(1))$ described in the proof below. The kernel of the  $\hat G$-action
on $\hat G/H$ is $\Zz(\hat G) \cap \Gamma$.\\
Case (ii) with $\hat G = SU (3) \times T^2$.  Then  $H$ is conjugate to $ SU (2)\cdot \Gamma$, where $SU(2)$  corresponds to the irreducible complex representation of $\h $ into $ \su (3)\subset \g$  of dimension 2 (case (2c)) and $\Gamma$ is a finite subgroup of $\Zz(SU(3))\times T ^2$.  The kernel of the  $\hat G$-action is $\Gamma$.\\
Case (iii)  with  $\hat G = Sp (1) \times Sp (1)\times Sp (1) \times U(1)$.  Then  $H$ is $ H^0 \cdot \Gamma$. Here $H ^0$ is the subgroup $Sp (1)$ diagonally embedded in $Sp(1)\times Sp(1) \times Sp(1)\subset \hat G$ (case (2a))  and $\Gamma$ is a finite  subgroup of $ \Zz (\hat G)$. 
The kernel of the $\hat G$-action is $\Z_2[\Zz (H^0)] \cdot \Gamma$. 
\end{proposition}

\begin{proof}  In  case (i)  direct computations  on Lie algebras show that  there are  only two  possible   (up to  a conjugation)
embeddings  $\so (3) \to \so (5)\subset gl (\R ^5)$ whose irreducible components are  of  real dimensions 3, 4 respectively. The first one has its adjoint representation  on $V$ as  a  sum of two real irreducible representations of dimension 3 and
one trivial representation, so it is case (2a). The corresponding pair of connected Lie groups is $(Spin (3)\subset Spin (5))$.

The isotropy 
representation of the second embedding of $\h $ into $\so(5)$ is  a sum of one real irreducible  representation of dimension 4 and three
real irreducible representations of dimension 1, so it is case (2c). The corresponding pair of connected Lie groups is 
$(Sp(1)\subset Sp (2))$.

We now examine which disconnected Lie subgroup $H$ in $G$   satisfies  the condition  of case (i). First let us assume that its identity connected component $H^0 = Spin(3) \subset Spin(5)=G$ satisfies the condition of case (ii), associated with possibility (2a). 
To find the normalizer $\Nn _{Spin (5)} Spin (3)$ we project it  into the group $SO(5)$.
The normalizer $\Nn_{SO (5)} (SO (3))$ is $S(O(2) \times O(3))$, according to the invariance principle. The group
$S(O(2) \times O(3))$ is generated  by $SO(2) \times SO(3)$ and  $D_{1,4}$,
moreover $SO(2) \cdot \Z_2 [D_{1,4}]$ is $\Zz_{SO(5)} (SO(3))$. Clearly $(Ad_{D_{1,4}})_{| V}$ belongs to $SO (4)_{3,4}$.
Let $H'$ be the image of the projection of $H$ on $SO(5)$.  
Then $H' = SO (3)\cdot \Gamma$, where $\Gamma \subset (SO(2)\cdot \Z_2 [D_{1,4}])$.  A direct calculation shows that the image of the adjoint 
action of $\Gamma$ on $V$ preserves  the $SO(4)_{3,4}$-invariant subspace $\R ^4 \subset V$, if and only if $\Gamma \subset \Z_2 [D_{1,4}]$. 
(Alternatively we compute that $\Nn_{SO(4)_{3,4}} ( SO (3)_{3,3}) = SO(3)_{3,3} \cdot \Z_2[a]$, where $a$ is the generator
 of the center $\Zz(SO(4))_{3,4}$, which gives us the same conclusion.) 
A direct computation gives the kernel of the action.

Now we assume that $H^0$  satisfies  the condition of  case (ii), associated with possibility  (2c). Using the invariance principle, we observe that the normalizer $\Nn_{Sp(2)} (H^0)$ 
is  $H^0 \times  Sp (1)_2$, where $Sp(1)_2\subset \Zz_{Sp (2)}(H ^0)$. Thus  $H $ is of the form $H^0 \times \Gamma$, where
$\Gamma$ is a  finite  subgroup in $Sp(1)_2$.  
We observe that the image of the adjoint representation of
$H^0 \times Sp(1)_2$ on $V$ coincides with the subgroup $SO(4)_{3,4} \subset \tilde G_2\subset Gl (V)$.
Thus the adjoint representation of $H$ lifts to an embedding of $\rho(H)$ into $SO (4)_{3,4} \subset \tilde G_2 \subset Gl (V)$. This proves Proposition \ref{Proposition.2.2.3}, case (i).

\medskip

In the second case (ii) the corresponding group  $\hat G$ is $SU(3) \times T^2$. A simple calculation using Lemma \ref{Lemma.2.2.2} shows that  there is only one (up to  a conjugation) Lie connected subgroup $H \subset \hat G$ such that $\h =\so(3)$, and the image of the isotropy  representation of the corresponding connected Lie
group $H^0$ is a subgroup of $\tilde G_2$. The group $H^0$ is  $ SU(2)\subset SU(3) \subset \hat G$ which  corresponds to the irreducible complex representation of $\h$  of dimension 2. Its isotropy
representation is a sum of  a real irreducible representation of dimension 4 and three trivial representations, so
it corresponds to case (2c).

To complete our examination of this case (ii) we need only to consider the case of a disconnected subgroup $H$.
Suppose that $H$ is a subgroup of $\Nn_{SU(3)\times T^2} (H^0)$ having
 $H^0$ as its identity connected component. According to the invariance principle, $\Nn _{SU(3)\times T ^2} (H^0)$ is $S(U(2) \times U(1)) \times T^2$.  Thus  $H$ has the form $H^0 \cdot  \Gamma$, where
 $\Gamma $ is a finite subgroup of $\Zz_{SU(3)}  (SU(2))\times  T^2$. Since the action of $\Gamma$ on $V$ has at least  three trivial components  of dimension 1, we conclude
 that $\rho (\Gamma)$ is a subgroup of $\rho (H^0)$. Hence  $\Gamma\subset \Zz (\hat G)\times T^2$. 
This proves Proposition \ref{Proposition.2.2.3}, case (ii).

\medskip

In the last case (iii) the corresponding group $\hat G$ is $Sp (1)\times Sp (1) \times Sp (1) \times U(1)$. 
Using Lemma \ref{Lemma.2.2.2} we conclude that any connected subgroup $H^0$ must be embedded diagonally into $Sp(1) \times Sp(1) \times Sp (1)$. It is easy to check that
the isotropy action of $H^0$ on $V$ is a sum  of two real irreducible representations of dimension 3 and one trivial representation of dimension 1, so it corresponds to case (2a).
 
Now we prove that any disconnected subgroup $H\subset \hat G$ satisfies the condition of case (iii), if  its identity  connected component  $H^0$
does.
Let us compute $\Nn_{\hat G} (H^0)$. Since  $Aut (H^ 0) = Int (H^0)$,  we have $\Nn _{\hat G} (H ^ 0) = H^ 0 \cdot \Zz_{\hat G} (H^0)$.  Clearly $\Zz_{\hat G} (H^0) = H^0 \cdot \Zz(\hat G)$.
Hence the image of $\Nn_{\hat G} (H ^0)$ under its isotropy action on $V$ is equal to  the image of the isotropy action
of $H^0$.  This completes the proof of Proposition \ref{Proposition.2.2.3}.
\end{proof}

{\it  Possibility 3 with $\h = \so (3) + \R$}. Lemma \ref{Lemma.2.2.2} implies that $\g$ cannot contain a commutative ideal of dimension greater than or equal  to 5.  
Since $\dim \g = 11$,  taking into account $\g \supset \h$, we conclude that   $\g$ is one of the following Lie algebras\\
i) $3\so(3)  +\R ^2$,\\
ii) $\so(5) + \R$,\\
iii) $\su (3) + \so (3)$,\\
iv) $\su(3) + \R^3$.

We exclude the last case (iv), since by Lemma \ref{Lemma.2.2.2} the adjoint representation of $\h$  on $V$ restricted to $\so(3)\subset \h$  has no  irreducible component of dimension 5, and if this representation has an irreducible  component of real dimension 4, the  other irreducible subspace has  real dimension 3.

Suppose that  $H_1$ and $H_2$ are connected Lie subgroups in a connected Lie group $G$ such that their Lie algebras $\h_1$ and $\h_2$ satisfy the condition $[\h_1, \h_2] = 0$. Then we denote by $H_1\cdot H_2$ the connected Lie subgroup  in $G$
whose Lie algebra  is  the direct sum $\h_1 + \h_2$. From Lemma \ref{Lemma.2.1.8} we get immediately

\begin{lemma}\label{Lemma.2.2.4}  In  group $SO(4)_{3,4}$ there is no subgroup  of the form
$SO(3) \cdot U(1)$. A  subgroup $SU(2) \cdot U(1)\subset Gl (\R ^7)$ corresponding to a representation $\Pi: \su (2) + \R \to \gl (\R ^7)$ can be seen as a subgroup of $SO(4)_{3,4}\subset Gl (\R ^7)$, if and only if one of the following two conditions (i) and (ii) is  fulfilled.\\
(i)  $\Pi$  
is a sum of one real irreducible component of dimension 4, corresponding to the highest weight $(1,1)$ on its Cartan subalgebra, and one  real irreducible component of dimension 2, corresponding to the highest weight $(0,1)$ on its Cartan subalgebra, and  one trivial component of dimension 1 (so   $\bar \rho(\h)$ is  in  situation (3b)).\\ 
(ii) $\Pi$  is a sum of one real irreducible component  of dimension 4, corresponding to the highest weight $(1,1)$ on its Cartan subalgebra,  and one real irreducible component of dimension 3, corresponding to the highest weight $(2,0)$  on its Cartan subalgebra (so  $\bar \rho(\h)$ is in  situation (3a)).
\end{lemma}

Using Lemma \ref{Lemma.2.2.4}  (or  Lemma \ref{Lemma.2.2.2})  we also exclude  the first case (i) of possibility 3 by looking at all possible  embeddings of
the summand $\h \subset \so(3)$ into $\g= 2\so(3) +\R^2$.  It remains to consider cases (ii) and (iii).

\begin{proposition}
\label{Proposition.2.2.5} Suppose that $\hat G/H$  admits a $\hat G$-invariant  $\tilde G_2$-structure such that their corresponding Lie algebras $(\h\subset \g)$ are  in possibility 3. Then one of the following cases happens.  (A detailed description of $H$ will be given in the proof.) 

In case (ii)  with $\hat G = Sp(2) \times U(1)_2$, the Lie subgroup $H$  is   $SU(2) \cdot U(1)_{k,l}\cdot \Gamma$ with $k\not = 0$, $(k,l) = 1$, and
$\Gamma$ is a finite subgroup of $U(1)_2$. The kernel of
the action is $\Z_2[\Zz(Sp(2))]\times \Gamma$. 

In case (iii) with $\hat G = SU(3) \times SU (2)$,  the Lie subgroup $H$ is one of the following  forms.\\
- $H = SU(2)_{2,0} \cdot U(1)_{k,l} \cdot \Gamma$,  where  $\Gamma$ is a finite subgroup in
 $\Zz (\hat G)$,  (so $\bar \rho(\h)$  is in case (3b)),  moreover  $kl \not=0$. We have $\Zz(\hat G) \cap H^0= Id$, if $(2k-3l)(4k-3l) \not = 0$. In general $\Zz(\hat G) \cap H$  can be any subgroup of $\Zz(\hat G) = \Z_3 \times \Z_2$ depending on $\Gamma$  and $k,l$.\\
- $H =SU(2)_{2,3}\cdot U(1)_{1,0}\cdot\Gamma$,  where $\Gamma \subset \Zz(\hat G)$, (so $\bar \rho(\h)$ is in case  (3a)). The kernel of the $\hat G$-action  is $\Z_2[\Zz(\hat G)\cap H^0)] \cdot \Gamma$.
 \end{proposition}

\begin{proof} Let us consider case (ii) with $\g = \so(5)+ \R$. We can assume that the projection $\Pi _1 (\R)$  of the summand $\R\subset \h$  on $\so(5)$ is
nonzero, otherwise the kernel of the isotropy action of $\h$ contains $\R$, and the action of $\bar \rho(\h)$ is not faithful.

A direct computation shows that the embedding of $\so(3)$ to $\so (5)$ is  associated with a real irreducible representation of $\so(3)$ of dimension  4 (complex dimension 2),  and the  projection $\Pi_1 (\h)$ is the Lie algebra of the centralizer $\Zz_{\so (5)} (\so (3))$.  A  subgroup $SU(2) \times U(1)\subset
 Sp (2) \times U(1)_2$  having this Lie algebra is  determined by 2 integers $(k, l)$ which are the coordinates of the component $U(1)$  w.r.t. $ U(1) _ {1 } \subset \Zz_{Sp (2) } (SU (2))$  and $U(1) _ 2$. We denote this  subgroup by $SU(2) \cdot U(1) _{k, l}$. By our condition
 $k \not= 0$ and  $(k, l) = 1$. We  check easily that the associated isotropy representation of $SU(2) \cdot U(1)_{k,l}\subset  Sp(2)\times U(1)_2$  corresponds to case (i) in Lemma \ref{Lemma.2.2.4}. 
 
Now let us find all Lie subgroups $H$ in  $\hat G$ satisfying the condition of Proposition \ref{Proposition.2.2.5}, case (ii). By our consideration above it follows that  the identity connected component $H^0$ of $H$ is embedded in $Sp(2)\times U(1)_2$ as $SU(2) \cdot U(1)_{k,l}$. 
Clearly $U(1)_2\subset \Nn_{\hat G}  (H^0) $. Using the invariance principle we  conclude that $\Nn_{\hat G} (H^0) = H^0 \times U(1)_2$.  This proves the first assertion  of Proposition \ref{Proposition.2.2.5}, case (ii). The second assertion
follows by a direct calculation.

Now let us consider case (iii)  with $\g = \su(3) + \so(3)$. Denote by $\Pi_1$ the projection of $\h$ on the summand $\su(3)\subset \g$ and by
$\Pi _2$ the projection of $\h$ on the summand $\so(3) \subset \g$. Using Lemma \ref{Lemma.2.2.4} we conclude that  $\Pi_1 (\R)$ is  nonzero, otherwise the restriction of  the isotropy action to the summand $\R\subset \h$ would have at least 5 trivial components. Repeating this argument, we conclude that  $\Pi_1 (\so (3))$ is also nonzero.
Clearly the embedding  of  $\Pi_1(\so(3))$ into  $\su(3) \subset \gl (\C ^3)$ must  correspond to its  complex irreducible representation of 
complex dimension 2, because its image commutes with  $\Pi _1 (\R)$. 
Hence the projection $\Pi _1 (\h)$  is defined uniquely up to   automorphisms of $\su (3)$.
Hence the embedding  of the component $ U(1)$ into $ SU (3) \times  SU(2) = G ^{sc}$ is  characterized by
two integers $(k, l)$ which are the coordinates of  $U(1)$ w.r.t.  $U(1)_1$ and $U(1)_2$, where $U(1) _1 = \Zz_{SU(3)} ( SU(2) )$ and
$U(1)_2$  being a  maximal torus of $SU(2)$. Further we  observe that there are two possible sub-cases. 

If $\Pi _2 (\so(3))$ is empty, then $k \not = 0$ and $l \not = 0$. Denote by $SU(2)_{2,0}\cdot U(1)_{k, l}$ the  connected Lie subgroup  of $SU(3) \times SU(2)$  having  Lie algebra $\h$ with this property. Its isotropy representation corresponds to case (3b) in Lemma \ref{Lemma.2.2.4}.(i). %Then 

If $\Pi_2(\so(3))$ is not empty, then $l = 0$, and hence $k = 1$.  Denote by $SU(2)_{2,3}\cdot U(1)_{1,0}$ the Lie subgroup of $SU(3) \times SU(2)$ having Lie algebra $\h$ with this  property. Its isotropy representation corresponds to
case (3a), see also Lemma \ref{Lemma.2.2.4}.(ii).

Now we consider  disconnected Lie subgroups $H$  whose Lie algebra $\h$ is in case (iii), the first sub-case (3b).  Denote by  the same $\Pi_i$ the lift of $\Pi_i$  from $\g$ to $\hat G$.  
Since $k \cdot l \not = 0$, we have
$$\Pi_2[\Nn _{\hat G} (SU(2)_{2,0}\cdot U(1)_{k,l})] \subset
\Nn_{SU(2)} \Pi_2 (U(1)_{k,l})  =\Pi_2 (U(1)_{k,l})\cdot \Z_2[A_{(12)}].$$
Here
\begin{math}
A_{(12)} = \left( \begin{array} {cc}
0 &\sqrt{-1}  \\
\sqrt{-1}&  0 \\
\end{array}
\right)\in SU(2).
\end{math}
 
But $Ad_{A_{(12)}}$ maps $U(1)_{k,l}$ to $U(1)_{k,-l}$. Now it is easy to see that $\Nn_{\hat G}( SU(2)_{2,0}\cdot U(1)_{k,l})=  SU(2)_{2,0}\cdot U(1)_{k,l} \cdot U(1) \cdot  \Zz (\hat G)$, where $U(1)\subset SU(2)\subset \hat G$.

Let $x\in  H \cap  U(1)$.  Since $x$ commutes with $H^0$, if $Ad_x$  belongs to  $SO(4)_{3,4}$, the image $Ad_x$  must belong to
 $Ad_{U(1)_{k,l}}$. Hence $x \in \Zz (\hat G)$.  This proves $H \subset H^0 \cdot \Zz(\hat G)$.
 A direct calculation gives  the kernel of the $\hat G$-action in this case. 
 
In the next sub-case (3a), using the invariance principle, we conclude that  $\Nn_{\hat G}(SU(2)_{2,3}\cdot U(1)_{1,0}) = SU(2)_{2,3}\cdot U(1)_{1,0}\cdot \Zz(\hat G)$.  A direct computation completes the proof of Proposition \ref{Proposition.2.2.5}.
\end{proof} 

{\it Possibility 4 with $\h = \R ^ 2$}.
If $rk\, \g \ge 4$, then the dimension of the fixed-point of the action of $\rho (H)$  on $V$ is at least 2 which
does not agree with the action of the maximal torus of $SO(4)_{3,4}$ on $\R ^7$.
Thus $ \g$ must be one of the following Lie algebras\\
i) $\so (3) + \so(3) + \so (3)$,\\
ii) $\su (3) + \R$.

In case (ii) instead  of working with  $\hat G = SU(3) \times U(1)$  it is more convenient to work with $G=U(3) = \hat G/ D_3$, where $D_3 = \{ (g, g ^{-1}) |  g = diag (e ^{ ({\sqrt{-1}2k\pi\over 3}}, e^{{\sqrt{-1}2k\pi\over 3}}, e^{{\sqrt{-1}2k\pi\over 3}}), k =1,2, 3\}$.
We note  that there is a 1-1 correspondence between  connected  Lie subgroups $H$ in $\hat  G$ and  connected Lie subgroups  $H'$  in $U(3)$  having  the same Lie algebra  $\h \subset  \g  = \su (3) + \R$.  Furthermore,     $\Nn _{\hat G}(H) = \pi ^{-1} (\Nn _{U(3)} ( H'))$, where $\pi : \hat G \to  U(3)$ is the natural projection. 
 Thus  it suffices to work with $G=U(3)$. As we will see    below, $\Nn_G(H')$  is  generated by   $H'$ and  a subgroup  $\Gamma \subset \Zz(G)$.   Hence, to get a full list of  a classification in  case (ii), working with $G=U(3)$ instead  of $\hat G$, we need 
examine  only one extra  possibility, if  the   corresponding connected Lie subgroup $H \subset  \hat G$ contains $D_3$.

\begin{proposition}\label{Proposition.2.2.6} 
In case (i) with $\hat G = SU (2) \times SU(2) \times SU(2)$, the Lie subgroup   $H$  is of the form $U(1) _{0, 1, - 1} \cdot U(1) _{1, 0, - 1}\cdot \Gamma$,
where $\Gamma \subset \Zz(\hat G)\times \Z_2[ (A_{(12)}, A_{(12)}, A_{(12)})]$. The kernel of the $\hat G$-action is $\Z_2[-Id, -Id, Id] \cdot ( \Gamma\cap \Zz (\hat G))$.\\
In case (ii) with $ G = U(3)$, the Lie subgroup $H$ is of the form of $U(1)_{k,k,k+1}\cdot U(1)_{m, m+1, m+1} \cdot \Gamma$, where $\Gamma \subset \Zz ( G)$. The kernel of the $ G$-action 
is $\Gamma$.\\
A detailed description of $H$ will be given in the proof below.
\end{proposition}

\begin{proof}
Let us fix  a  subgroup $SO(2)_{2,2} \subset SO(3)_{3,3}\subset SO(4)_{3,4}$. We can choose a subgroup $U(1)_{0,4} \subset SU(2)_{0,4} \subset SO (4)_{3,4}$ such that  these subgroups
are generators of  a maximal torus    of $SO(4)_{3,4}$. 
Any  subgroup $U(1)$ in this torus shall be  denoted by $U(1)_{p,q}$   with respect to this lattice.

In case (i) let us fix a  maximal torus $ U(1) _1 \times U(1)_2 \times U(1)_3 $ of $\hat G = SU(2)_1 \times SU(2)_2 \times SU(2)_3$ such that $U(1)_i \subset SU(2) _i$.  Let $ T ^2 $ be a torus  in  $\hat G$  such that $\rho (T^2) \subset 
SO(4)_{3,4}$. W.l.g.  we can assume that $\rho ( T ^2) = U(1)_{0,4} \cdot SO(2)_{2,2}$. Let $U(1)_{k,l, m}$ be the preimage $\rho ^{-1} (U(1)_{p,q})$, where $(k,l,m)$ are the coordinates  with respect to $U(1)_i$. The weights of the  adjoint action of $\rho ^{-1} (U(1)_{p,q})$ on $V$ are $(1, \exp\pm 2k\sqrt{-1}\theta, \exp\pm 2l\sqrt{-1}\theta, \exp \pm 2m\sqrt{-1}\theta)$
which must coincide with the  weights of the  representation of $U(1)_{p,q}$ on $\R^7$ which are 
$(1,\exp \pm \sqrt{-1}p\theta, \exp\pm \sqrt{-1}q\theta,\\
\exp\mp \sqrt{-1}(p +q)\theta)$. Taking into account that the isotropy action of $U(1)_i$ on $V_i \subset \su(2)_i$ is a double covering,  we conclude that $k = \pm p, l = \pm q, m= \mp (p +q)$.   Each choice of the sign  of the weights
of the action of the torus on $V ^7$ corresponds to a different solution  of the coordinates $(k,l,m)$  of $T^2$.  Observing that $T^2$  is  invariant under  the  inverse map $ x\mapsto x ^{-1}$, we have  actually only  four different
solutions for the coordinates $(k,l,m)$. Using the  permutations between $SU(2)_i$, we get only  three different solutions for  $T^2$:\\
 $T ^2_1 = U(1) _{0, 1, -1} \times U(1)_{1, 0, -1}$, $T^2_2 = U(1) _{0,1,1} \times U(1)_{1,0,1}$, $T^2_3 = U(1)_{0,-1, 1} \times U(1)_{1,0,1}$.\\

It is easy to see that $T^2_2  $ and $T^2 _3$ are equivalent up to  automorphism of $\hat G$.
Since we can change the orientation of each $U(1)_i \subset SU(2) _i$,  the  tori $T^2_1$ and $T^2_2$ are  equivalent.
Thus  up to   conjugation by automorphism of $\hat G$, there is only one choice of $T^2$  satisfying our condition.

To complete  the examination of case (i) we need to find all disconnected  Lie  subgroup  $H$ whose  identity connected Lie component $H^0$
is the  torus $U(1) _{0,  1, -1} \cdot U(1) _{1, 0, - 1}$.  Clearly $T^3 = U(1)_1 \times U(1)_2 \times U(1)_3 \subset \Nn_{\hat G} (H^0)$. Considering the projection of $\Nn_{\hat G}(H^0)$ on each  factor $SU(2)_i$ we conclude that $\Nn_{\hat G} (H ^0) \subset T^3 \cdot (\Z_2[A_{(12)}]) ^3 $.

A direct  calculation  shows that $\Nn_{\hat G} (H ^0) = T^3 \cdot \Z_2[(A_{(12)},  A_{(12)}, A_{(12)})]$.
Hence $H = H^0 \cdot \Gamma$, where $\Gamma$ is a finite subgroup in $ T^3 \cdot \Z_2[(A_{(12)}, A_{(12)}, A_{(12)})]$.
Further we note that the element $Ad_{( A_{(12)}, A_{(12)},A_{(12)})}$ in $SO(V^7)$
is  $ D_7 = diag (-1, 1, -1, 1-1, 1, -1)$, which belongs to $SO(4)_{3,4}$.

Clearly the element $Ad_x$ in $SO(V^7)$, where  $x\in T ^3$,  belongs to $SO(4)_{3,4}$, if and only if $x \in H^0 \cdot \Zz(G)$.
 Thus  the image  $\rho(H ^0 \cdot \Gamma)$ belongs to $SO(4)_{3,4}$,  if and only if $\Gamma \subset \Zz(G) \cdot \Z_2[(A_{12}, A_{(12)}, A_{(12)})]$. A direct computation  yields the second statement of Proposition \ref{Proposition.2.2.6}.

It remains to consider case (ii) with the corresponding group $ G = U(3)$.  
   Now we use
the notations $U(1)_1 , U(1)_2, U(1) _ 3$ for the  generators of the maximal torus  of $U(3)$.  Suppose that
there is $T ^2 \subset U(1)_1 \times U(1) _2 \times U(1)_3$ such that $\rho ( T^2) = U(1)_{0,4}\cdot SO(2) _{2,2} \subset
SO(4) _{3,4}$.  The  weights  of  the isotropy action of $U(1)_{k,l,m}$  is $(1, \exp\pm  \sqrt{-1}(k-l), \exp \pm (l-m), \exp \sqrt{-1}(k-m))$ and
the weights of the  representation of $U(1)_{p,q}$ are $(\exp \pm  p, \exp\pm q, \exp\mp (p +q))$.
Thus $T^2 $  must be $U(1)_{k, k, k +1}\cdot U(1)_{m,m+1, m+1}$ or
$U(1)_{k,k,k-1}\cdot U(1)_{m,m-1, m-1}$. These  two families of solutions are actually   mirror  identical.

Now let us find all disconnected Lie group $H$ whose identity component $H^0$
 is conjugate to  the torus $T ^2_{k,m} =U(1)_{k, k, k+1} \cdot U(1) _{m, m+1, m +1}$.  Since $\h$ contains a  regular element, it follows that  the identity connected component of $\Zz_{\hat G} (H^0)$ is a torus $T ^3$. Denote by $lT^3$ the Lie algebra of  $T^3$. Using the invariance principle
applying to $W _1 ^\perp = lT ^3 \subset \g$, we conclude that $\Nn _G (H ^0)$  leaves $T^3$ invariantly. Hence
$\Nn_{\hat G} (H ^0)$ is a subgroup of
$\Nn_{\hat G }( T ^3)  = T ^3 \cdot \Sigma_3$, where $\Sigma _3$ is the Weyl group generated by   two elements  of
order 3 and of order 2 in $SU(3)$.

Since $(Ad_{ \Sigma _3})_{| \g} \subset  SO (\g)$, an element  $x\in \Sigma _3$ belongs to the normalizer $\Nn_G ( H^0)$,  if and only if it  leaves the orthogonal complement $\langle(-(m +1), (m- k) , k )\rangle _\R$ of  $\h $ in $lT ^3= \R ^3$ invariantly.   The generators of $\Sigma_3$   are

\begin{math}
A_{(123)}  =  \left(\begin{array}{ccc}
0 & 0 & 1\\
1& 0 & 0\\
 0& 1 & 0 \\
\end{array}\right)\text{ and } B_{(23)} = \left  (\begin{array}{ccc}
-1 & 0 & 0\\
 0 & 0&   1\\
 0 & 1 & 0\\
\end{array}\right).
\end{math}

They act on $T ^3$ by permuting coordinates $k,l,m$. 
We conclude that
$\Nn_{\hat G} ( H ^0) = T ^3$, if $T^2_{k,m}$ is regular, i.e. if  all three coordinates  $-(m+1), (m-k), k$  are mutually different.  If $T_{k,m}$ is singular, $\Nn_{\hat G} ( H^0) = T ^3 \cdot \Gamma_0$, where  $\Gamma _0 \subset \Sigma_3$ and $Ad_{\Gamma _0}$ permutes two   equal  coordinates of $(-(m +1), (m- k) , k )$.
Arguing as in case (i), we conclude that  for regular tori $T^2_{k,m}$ we have
$H = H ^ 0 \cdot \Gamma$, where $\Gamma \subset \Zz (G)$. In this case
$\Zz (G) \cap H = \Gamma$.   For a singular torus $T^2_{k,m}= H^0$ we need  also to consider   the case, when $H$ contains an element  of $\Sigma _3$.  A direct computation shows that the action $Ad_x$, $x\in \Sigma_3,$ 
permuting two coordinates  of $T ^3$ acts on the invariant subspace $\R ^3 \subset \R ^7$ as $ (1, 1, -1)$, hence
it does not belong to $SO(4)_{3,4}$. Thus this case cannot happen. This completes our proof.
\end{proof}

{\it Possibility  5 with $\h = \R $}. Clearly  $rk\, \g\le 5$, since the action of any
group $U(1) \subset\tilde G_2$ on $ \R^7$ is non-trivial. Since $\dim \g = 8$, we conclude that    $\g$ is
one of the following Lie algebras:\\
i) $2\so(3) + \R^2$,\\
ii) $\su (3)$.

The Lie group $\hat G$ with Lie algebra $2\so(3)+\R^2$ is isomorphic to $SU(2)\times SU(2) \times T^2$.
By the same argument as in  our consideration  of possibility 4, case (ii),  we  can work equally with
the   group $U(2)\times U(2)$ instead  with $\hat G$. To distinguish the isomorphic 
factors $U(2)$ in this decomposition of $ G$, we denote them by $U(2)_1$ and $U(2)_2$.

\begin{proposition}\label{Proposition.2.2.7} In case (i)  with $ G=U (2)_1 \times U (2)_2$,  the Lie subgroup $H$ is  $U(1)_{k, k+1,l, l+1}\cdot \Gamma$, where $\Gamma$ is a finite subgroup of $\Zz ( G) = U (1)_1 \times U(1)_2$.  The kernel of this action is $\Gamma$.\\
In case (ii) with $\hat G = SU(3)$ the Lie subgroup  $H $ is  $U(1) _{k,l,m}\times \Gamma$, where  $(k,l)= 1$,  and
$\Gamma$ is any finite subgroup of the maximal torus $T^2 \subset SU(3)$.  If $k = l = 1$, then $H$ can also take the form $U(1)_{1,1,-2}\cdot \Gamma$,
where $\Gamma$ is a finite subgroup in $SU(2)$. If $k\not= l$, the kernel of the  $\hat G$-action is either $Id$ or $\Zz (\hat G)$, depending on $\Gamma$.  If $k = l=1$, the kernel of the  $\hat G$-action is $\Zz(\hat G) = \Z_3$.\\
A detailed  description of $H$ will be given in the proof.
\end{proposition}

\begin{proof}
Let us consider case (i) with $G = U(2)_1 \times U(2)_2$.
Any embedding $ U(1) = \exp \h $ into $G$ is characterized by  a quadruple  of integers $(k_1,k_2,l_1,l_2)$ which
are coordinates of $U(1)$ in $ U(1)_{11} \times U(1)_{12} \times U(1)_{21} \times U(1)_{22}$, where
$U(1)_{ij}\times U(1)_{ii} $  is a maximal torus of $U(2)_i$. We denote  by $U(1)_{k_1,k_2,l_1,l_2}$ this subgroup $\exp \h$. The isotropy  action of $U(1)_{k_1, k_2, l_1, l_2}$  with parameter $\theta$  has weights $(\exp \pm \sqrt{-1}(k_1-k_2) \theta, \exp\pm \sqrt{-1} (l_1 - l_2) \theta, 1, 1, 1)$.  Note that $\rho (U(1)_{k_1,k_2,l_1,l_2})$ 
can be  written as $U(1)_{p,q}$ as in the proof of  Proposition \ref{Proposition.2.2.6}, case (i). Since the weights
of the representation of $U(1)_{p,q}$ on $\R^7$ are  $(1, \exp \pm \sqrt{-1}p\theta, \exp \pm\sqrt{-1}q \theta, \exp \mp \sqrt{-1}(p+q) \theta)$ coincide with  the weights of the isotropy action of $U(1)_{k_1, k_2, l_1, l_2}$, we conclude that
$U(1)_{p,q}$ must be either $U(1)_{0,4}$ or $SO(2)_{2,2}$ (cf. with the proof of Proposition \ref{Proposition.2.2.6}, case (i)). 
So  $k_1 - k_2 = \pm 1$ and $l_1 - l_2 = \pm 1$.  Up to automorphism of $G$  all these  solution  subgroups
are equivalent, so  we  will take a representative $U(1)_{k, k+1, l, l+1}$ of these solutions.

We compute $\Nn_G ( U(1)_{k,k + 1,l, l+ 1})$ easily, by using the projection of this subgroup on each component 
$U(2)_i \subset G$. Knowing   $\Nn _{SU (2) } (U (1)) = U(1) \cdot \Z _2[A_{(12)}]$ 
we conclude that $\Nn_G  ( U (1)_{k, k+1, l, l+ 1}) = T ^4$, if $(k +1) ^2 + (l + 1) ^2 \not = 0$.
Otherwise $\Nn_G ( U(1) _{-1, 1,-1,1}) = T^4 \cdot \Z_2 [ (A_{(12)}, A_{(12)})]$.

In the   first case  
$H = U(1)_{k, k +1, l , l+1 } \times \Gamma$, where $\Gamma$ is  a finite subgroup of $T^4$.
Since the isotropy action of $(\exp \sqrt{-1} \theta _1, \exp \sqrt{-1}  \theta _2 , \exp \sqrt{-1}  \tau_1, \exp \sqrt{-1} \tau_2)$   acts on the  fixed-point subspace $\R ^3\subset \R^7$  of $\rho(H^0)$ as the identity,
we  conclude that
$(Ad_{\Gamma })_{ | V } \subset Ad_{U(1)_{k, k +1, l , l+1 }}$, hence $\Gamma \in H^0 \cdot \Zz (G)$.

In case $k =1 = l$, a direct computation shows that  the action of $Ad_{  (A_{(12)}, A_{(12)})}$  changes  orientation of $V^7$. Thus  the examination of this case can be done as in   the  previous case with $k\not = l$. This proves  the first assertion of Proposition \ref{Proposition.2.2.7}.
The second assertion  follows  by direct computation. 

Now let us consider case (ii). An embedding $\exp \h = U(1)   \to   T ^2 \subset SU (3)= G ^{sc}$
can be characterized by  a triple  $(k, l, m)$  with $k+l+m = 0$ and $(k, l) = 1$. We  denote
this subgroup by $U(1)_{k,l,m}$. The weights of the isotropy action of $\h$ on $V$ are
$(0, \pm \sqrt{-1} (k-l), \pm \sqrt{-1} (l-m), \pm \sqrt{-1} (m-k))$. 
The group $ \rho ( U(1))$ can be  embedded into $SO(4)_{3,4}$  by setting the coordinates $p, q$ of this subgroup
$ \rho (U(1))$ in the maximal torus $ T ^2 $ of  $SO(4)_{3,4}$ whose basis  is  subgroups $U(1)_{0,4} \subset  SU(2)_{0,4}$ and $SO(2)_{2,2}\subset SO(3)_{3,3}$  as above. 
Since the weights of the  action of  $U(1)_{p,q}$ on $\R ^7 $ are
$1,  \exp \pm \sqrt{-1}p\theta, \exp \pm \sqrt{-1} (-p +q)\theta,  \exp \mp  \sqrt{-1} q \theta$,  we have
$ p = (k -l), \, -p + q = (l -m), q = k -m$. If $k \not = l$,  then $\ker \rho_{U (1)_{1,1,-2}} = \Z_3$.
 
To compute the normalizer $\Nn_{SU(3)} (U(1)_{k,l,m})$, as in the previous case, 
we observe that the connected  component of $\Zz_{SU(3)} (U(1)_{k,l,m})$ is $T^2$. Applying the invariance
principle, we conclude that $\Nn_{SU(3)}( U(1) _{k,l,m})$ leaves the torus  $T^2$ invariantly. 
Hence  $\Nn _{SU(3)} (U(1)_{k,l,m})$ is a subgroup of the normalizer $\Nn _{SU(3)} ( T ^2) = T ^2 \cdot \Sigma _3$. Arguing as in  possibility 4, case (ii), we conclude that  an element $x\in \Sigma _3$  normalizes   $U(1)_{k,l,m}$, only if $x = Id$, because $(k,l,m)$ is  regular.  Thus $\Nn_{SU(3)} (U(1)_{k,l,m}) = T^2 $, for $(k,l) = 1$ and $k \not = l$.  It is known that $\Nn_{SU(3)} U(1)_{1,1, -2} = SU(2)\cdot U(1)_{1,1, -2}$.

Now let us consider a disconnected Lie subgroup $H$ whose identity component $H^0$ is $U(1)_{k,l,m}$. Clearly  $H= H^0 \times \Gamma$, where $\Gamma$ is a subgroup of the maximal
torus $T^2 \subset SU(3)$. 
The same argument as in the previous case implies that the image of $Ad_{T^2}$ is the maximal  torus of
$SO(4)_{3,4}$. This proves the third assertion of Proposition \ref{Proposition.2.2.7}. Applying Lemma \ref{Lemma.2.2.4}.ii we prove the assertion for the case $k = l =1$.  A  direct computation
of $\Zz(G)\cap H$ completes the proof of  Proposition \ref{Proposition.2.2.7}.
\end{proof}

{\it Possibility 6 with $\h = 0$}. In this case $H$ is a finite subgroup of a compact group $ G $ dimension 7. Thus    $\hat G$ is  one of the following cases:\\
6i) $T^7$,\\
6ii)  $SU(2)\times T^ 4$,\\
6iii) $SU(2)\times SU(2)\times U(1)$.

Clearly  any group  $G$ listed above admits a  $G$-invariant 3-form of $\tilde G_2$-type. Since $T^7$ is commutative,
wee need only to verify in case (6ii) (resp. case (6iii)), whether  there is a finite non-central subgroup $H\subset G$ such
that $\rho(H)\subset SO(3)$ (resp. $\rho (H) \subset SO(3)\times SO(3)$) is a subgroup of  $SO(4)_{3,4} \subset G_2$. 
In  case (6ii) the action of any element $e\in \rho (H)$ on $\R^7$  leaves  a subspace $\R^ 5$  invariant. On the other hand, any   element  $e\not = Id \in SO(4)_{3,4}$  
is conjugate to an element in  $T^2 \subset SU(3) \subset SO(6) \subset SO(7)$, which cannot have its fixed point  subspace in $\R^7$ of dimension greater than 3.
Thus $\rho (H)$ consists only of the identity.
In case (6iii), let $SO(4)_{3,4}$ be  a maximal compact subgroup in $\tilde G_2$ containing  $\rho(H)$, whose existence follows from \cite[Theorem 2.1, p. 256]{Helgason1978} (see also Lemma \ref{Lemma.2.1.6}).  We  note that $\rho (H)$ is a subgroup of $SO(3)\times SO(3)$ as well as a subgroup of
$Gl (\R^6) \cap G_2 = SU(3)$, (this is a consequence of the transitivity of the $G_2$-action on $S^6 \subset \R^7$ \cite{Bryant1987}, or  see \cite[\S 2]{CS2002} for an alternative argument), taking into account that $SO(4)_{3,4}$ is also a subgroup of $G_2$ by Lemma \ref{Lemma.3.1.3}. Let $V_1 = \R ^3$ and
$V_2 = \R^3$ be  invariant subspaces of  $\rho(H)$ and $J$ be  the complex structure on $\R^6$. There are two possibilities: either $V_2 = JV_1$, or
$JV_1 \cap V_1  = \R ^2_1 $  and $JV_2 \cap V_2 = \R^2_2$. In the first possibility  $\rho (H) $ is a subgroup
of $(SO(V_1) \times SO(V_2))  \cap SU(3) = SO(3)_{3,3}$.   In the second possibility $\rho (H)$ is a  cyclic
subgroup  of the form $(x, x^{-1}) \in SO(3) \times  SO(3)$.  Clearly these subgroups   belong to $SO(4)_{3,4}$.    
Thus we get

\begin{proposition}
\label{Proposition.2.2.8} i) Let  $H$  be a finite  subgroup of $T^7$. The  action of $T^7$ on $T^7/H$ is effective, iff $H = \{e \}$.\\
  ii) Let $H$ be a finite   subgroup of a  compact Lie  group $G = SU(2) \times  T^4$.  The  quotient  space $G/H$  admits a $G$-invariant 3-form of $\tilde G_2$-type, if and only if $\rho(H)$ is central.\\ 
 iii) Let $H$ be a finite  not central subgroup
of a compact Lie group  $G = SU(2) \times SU(2) \times U(1)$. The  quotient  space $G/H$  admits a $G$-invariant 3-form of $\tilde G_2$-type, if and only if  $\rho(H)$ is a  subgroup of $ SO(3)_{3,3}$ or a cyclic group of the form
$(x, x^{-1})\in SO(3) \times SO(3)$.\\ 
\end{proposition}

\subsection{Classification theorem}
In this subsection we summarize our computation in the   previous  subsection  in the following Theorem \ref{Theorem.2.3.1}, taking into account  our remarks  before Propositions \ref{Proposition.2.2.6} and 
\ref{Proposition.2.2.7}. We also provide a formula to compute the dimension of the space of all invariant $\tilde G_2$-structures on a given  manifold $G/H$, see Remark \ref{Remark.2.3.2}.d.

\begin{theorem}\label{Theorem.2.3.1} Let $G/H$ be a homogeneous space admitting
a $G$-invariant $\tilde G_2$-structure.  We assume that $G$ is a connected compact Lie group  and $G$ acts  effectively on $G/H$. Then $G/H$ is  one of the following spaces
\end{theorem}

\begin{math}
\begin{array}{ccc}
Case & G & H  \\ %& d_3 (G/H)= d_1 + d_2 \\
1 & (Sp(2) \times Sp(1)) /\Z_2 & SO (4)_{3,4} \\ %& d_1 = 0,\, d_2 =2\\
1 & SO(5) \times SO (3) & SO(4)_{3,4} \\ %  & d_1 = 0,\, d_ 2 = 2 \\
2ai,\,\Gamma \subset \Z_2 \times \Z_2 & SO(5) & SO(3)\cdot \Gamma   \\%&  d_1 =1,\, d_2 = 4\\
2ci,\Gamma\subset Sp(1) & Sp(2) & Sp (1) \times \Gamma\\ % &  d_1 = d_1 (\Gamma_{|\R ^3})\\
%  &                   &                      & d_2 =1 + d_2 (\Gamma_{|\R ^3})  \\
2ci,\Gamma\subset SO(3) & SO(5) & Sp (1) \times \Gamma \\ %&  d_1 = d_1 (\Gamma_{|\R ^3})\\
%  &                   &                      & d_2 =1 + d_2 (\Gamma_{|\R ^3})  \\
2cii & SU(3) \times T^ 2 & SU(2) \\
2cii & PSU(3) \times T^ 2 & SU(2) \\%& d_1 = 3, \, d_2 = 4 \\
2aiii & (Sp(1)\times Sp(1) \times Sp(1))/\Z_2 \times U(1)  &  SO(3)  \\% &  d_1 =1,\, d_2 = 4\\
2aiii & SO(3) \times SO(4) \times  U(1)  &  SO(3) \\
2aiii & SO(3) \times SO(3) \times SO(3) \times U(1) & SO(3)\\
3bii,\, (k,l)= 1, \, k\not = 0&  SO(5) \times U(1)  & SO(3) \cdot U(1)_{k,l}\\% &  d_1 = 1, \, d_2 = 3\\
3biii,\, (k,l)= 1, kl\not=0 & SU(3) \times SU(2) & SU(2)_{2,0}\cdot U(1)_{k,l}\\
3biii,\, (k,l)= 1, kl\not=0 & SU(3) \times SO(3) & SU(2)_{2,0}\cdot U(1)_{k,l}\\
3biii,\, (k,l)= 1, kl\not=0 & PSU(3) \times SU(2) & SO(3)_{2,0}\cdot U(1)_{k,l}\\ % d_1 = 1,\, d_2 = 3  \\
3biii,\, (k,l)= 1, kl\not=0 & PSU(3) \times SO(3) & SO(3)_{2,0}\cdot U(1)_{k,l}\\
3aiii &   SU(3)\times SO(3)&      SO(3)_{2,3}\cdot  U(1)_{1,0}\\ %
3aiii & PSU(3) \times SO(3) &SO(3)_{2,3}\cdot  U(1)_{1,0}\\ %
%)  &  d_1 = 0, \, d_2 = 2 \\
4i &(SU(2) \times SU(2) \times SU(2))/\Z_2&  T^2 \text{ or } T^2 \cdot \Z_2\\ % & d_1 =3, d_2 = 4 \\
4i & SO(3)  \times SO(4) &   T^2 \text { or } T^2 \cdot \Z_2\\%& d_1 =3, d_2 = 4\\
4i & SO(3) \times SO(3) \times SO(3)&  T^2\text { or } T^2 \cdot \Z_2\\% & d_1 =3, d_2 = 4\\
4ii &  SU(3) \times U(1) & U(1)_{k,k,k+ 1} \cdot U(1)_{m, m+1, m+1}\\
4ii & U(3)   &  U(1)_{k,k,k+ 1} \cdot U(1)_{m, m+1, m+1} \\
%4iv & U(3) &  U(1)_{k,k,k-1} \cdot U(1)_{m, m-1, m-1} \\%&  d_1 = 1,  d_2 = 2\\
4ii  & PSU(3)\times U(1) &  U(1)_{k,k,k+ 1} \cdot U(1)_{m, m+1, m+1} \\
%4iv & PSU(3)\times U(1) &  U(1)_{k,k,k- 1} \cdot U(1)_{m, m-1, m-1} \\
5i & SU(2) \times SU(2)\times U(1)\times U(1) & U (1)_{k,k+1, l, l+1}\\
5i & SU(2) \times U(2) \times U(1) & U (1)_{k,k+1, l, l+1}\\
5i & U(2) \times U(2) & U (1)_{k,k+1, l, l+1}\\% & d_1 = 3, d_2 = \\
5i & SO(4) \times U(1)\times U(1) & U(1)_{k, k+1, l, l+1 }\\
5i & SO(3) \times SO(3)\times U(1) \times U(1) & U (1)_{k,k+1, l, l+1}\\%      &                &          &  d_2 = 4 \{ (k_i = k_j, i \not = j \}\\
5 ii,\, (k,l)= 1,\, k\not = l & SU(3) &  U(1)_{k,l} \cdot \Gamma,  \Gamma \subset U(1)\\
5 ii,\, (k,l)= 1,\, k\not = l & PSU(3) &  U(1)_{k,l} \cdot \Gamma,  \Gamma \subset U(1)\\% & d_1 = 1,\, d_2 =5 \\
5ii,\, k = 1= l & PSU(3) & U(1)_{1,1}\cdot \Gamma, \Gamma \subset SU(2)   \\
%& d_1 = 1 ,\, d_2 = 5 \\
6i &  T^7 &    \{ e\}\\
6ii &  SU(2) \times T^4&     \rho(H) = \{ e\}  \\   %d_1 = 7,  d_2 = 28.\\
6ii &  SO (3) \times T^4&    \rho(H) =  \{e \}  \\
6iii,\, \#(H)<\infty & SU(2) \times SU(2) \times  S^1  & \rho( H) \subset SO(3)_{3,3}, \text{ or } \rho(H) = \Z_k \\   %d_1 = 7,  d_2 = 28.\\
6iii,\, \#(H)<\infty&  SO (3) \times SU(2) \times S^1&    \rho(H) \subset SO(3)_{3,3}, \text{ or } \rho(H) = \Z_k\\
6iii,\, \#(H)<\infty&  SO (3) \times SO(3) \times S^1&     \rho(H) \subset SO(3)_{3,3},\text{ or } \rho(H) = \Z_k
\end{array}
\end{math}

In this table,  spaces   have the same  covering, if  and only if  they have
the  same  numeration.  We also use the  notation $PSU(3)$ for the quotient $SU(3)/\Zz (SU(3))$.

We now define the degree of rigidity of  $G/H$   as the     dimension of the space of all $G$-invariant  3-forms
of $\tilde G_2$-type  on $G/H$,  and  we denote this degree by $d_3(G/H)$.  This dimension is equal to the dimension  of the space of  all  $G$-invariant
3-forms on $G/H$, since the $GL(\R ^7)$-orbit of $\tilde \phi$ is open in $\Lambda ^3 (\R ^7) ^*$. Hence the degree of rigidity of $G/H$ equals  the dimension of the space of  all $\rho(H)$-invariant  3-forms on $V$. 
 
We  have the following 
decomposition (see e.g. \cite{Bryant1987}, \cite[table 5]{VO1988})
$$ \Lambda ^3 ( V ^*) = \Lambda^3_1 (V ^*)\oplus \Lambda ^3 _{7}(V ^* ) \oplus \Lambda ^3_{27} (V ^*),$$
where $\Lambda ^3 _i(V ^*)$ is the component of dimension $i$.  The component $\Lambda ^3 _1 $ is generated
by $\tilde \phi$, the component  $\Lambda ^3 _7$ is $\tilde G_2$-isomorphic (and hence $\rho(H)$-isomorphic) to  $V ^* = V $, and  $\Lambda ^3_{27}$
is $\tilde G_2$-isomorphic (and hence $\rho(H)$-isomorphic) to the space $S ^2 _0 (V ^*)$ of traceless quadratic  forms
on $V$.
This isomorphism can be  written explicitly as \cite[(2.15)]{Bryant2005}
$$ i_{\tilde \phi}  (\alpha \circ \beta) = \alpha \wedge *_{\tilde \phi} (\beta\wedge *_{\tilde \phi} \tilde \phi) + \beta \wedge *_{\tilde \phi}(\alpha \wedge *_{\tilde \phi} \phi).$$
Now let $\rho(H)$ be a subgroup of $SO(4)_{3,4}\subset \tilde G_2 \subset Gl (\R ^7)$. Denote by $d_1$  the dimension of
the  fixed-point subspace of $V$ under the action of $\rho(H)$.  Denote by
$d_2$  the  dimension  of the subspace  of all $\rho(H)$-invariant quadratic  forms on  $V$. 
Then we have
\begin{equation}
d_3 (G/H) =  d_1  + d _2. 
\label{2.3.2}
\end{equation}
Dimension $d_1$ is already explicit from the embedding $\rho : H \to SO (4)_{4,3}\subset  Gl (\R ^7)$. To compute $d_2$ we use the decomposition
$S ^2 (\rho)$ computed in \cite[table 5]{VO1988}.

\begin{remark}\label{Remark.2.3.2}
a) Since $SO(4)_{3,4}$ is also a compact Lie subgroup of $G_2$, all of  the   homogeneous spaces
$G/H$  listed above also admit $G$-invariant $G_2$-structures.  Hence the dimension of  the space of all $G$-invariant 3-forms
on $G/H$ is at least 2. \\
b) Some  different  spaces $G/H$ listed above are diffeomorphic as  differentiable manifolds, e.g. $(Sp(2)\times Sp(1)/\Z_2)/SO(4)_{3,4}$ (case 1)  and $Sp(2)/Sp(1)$  (case  2ci) are diffeomorphic to  the standard sphere $S^7$.  Other examples are the Wallach spaces in  (5ii)  with different (k,l). We refer the reader to \cite[p. 466]{KS1991} for a precise formulation, when  these Wallach spaces are diffeomorphic.\\
c)  As a consequence of our classification we get a  new proof for a statement in \cite{Le2006a} that $S^3 \times S^4$ admits no 
homogeneous $\tilde G_2$-structure. Since $S^3 \times S^4$ is simply connected,  by \cite[Theorem A]{Montgomery1950} if  $S^3 \times S^4$ admits
a transitive action of a group $G$ it admits also a  transitive action of a  compact Lie subgroup $G' \subset G$.
On the other hand $S^3 \times S^4$ is not in our list.\\
d) Clearly the dimension of the space of $G$-invariant  $\tilde G_2$-structures  on $G/H$ is equal to $d_3(G/H) -1$.
\end{remark}

\section{Compact homogeneous  manifolds  admitting invariant $G_2$-structures}

In this section we classify all homogeneous spaces $G/H$ admitting a $G$-invariant $G_2$-structure such that $G$ is a compact Lie group and $H$ is a closed
Lie  subgroup  (not necessary connected) of $G$. Our strategy is similar to  that one in the previous section.  We also compute
the dimension of  the space of $G$-invariant  $G_2$-structures on $G/H$, see Remark \ref{Remark.3.3.2}.a.

\subsection{Reduction to a representation problem}

We use the same method as in  the previous section  to classify  all pairs $(H\subset G)$ of a compact Lie group  $H$ in a compact Lie group $G$
such that $G$ acts effectively on $G/H$ and $G/H$ admits a $G$-invariant $G_2$-structure.  First we will  classify all pair of the
corresponding Lie algebras $(\h\subset \g)$  such that $\bar \rho (\h) \subset \g_2$.  Combining the list of maximal Lie subalgebras in
$\g_2$  and  the list of Lie compact subalgebras in $\so(4)_{3,4}$ in the previous section we  get the  following list of  compact 
Lie algebras  $\bar \rho (\h)$ in $\g_2$\\
1) $\bar\rho( \h) = \so (4)_{3,4}$;\\
2) $\bar\rho(\h) =  \so(3)$ with four possible embeddings into  $\g_2$.
In the first three cases (2a), (2b), (2c) we have $\bar \rho (\h)\subset\so(4)_{3,4}$, see also subsection 3.1.
In the last case (2d)   we have $\bar \rho(\h) = \so (3)_7$;\\
3) $\bar\rho(\h) =  \so(3)+ \R\subset \so(4)_{3,4}$,\\ 
4) $\bar\rho( \h) =  \R ^2 $; \\
5) $\bar\rho( \h) =  \R ^1 = \so (2) $ (there are infinitely  many inequivalent embeddings of  $\so (2)$ into $\so (4)$)\\
6) $\bar\rho(\h) =  0 $;\\
7) $\bar\rho(\h) = \g_2$;\\
8) $\bar \rho(\h) = \su(3)$.

The first cases 1-5, except  cases (2d), have been analyzed on the algebra level in  the previous subsection 3.1.  When lifting to  the  corresponding
Lie subgroup we  need to check whether  the corresponding disconnected Lie subgroup $H$ belongs to $G_2$ but does not belong to
$SO(4)_{3,4}$. Further, we notice that  any subalgebra $\su(2)$ in $\su(3) \subset \g_2$ is conjugate to $\su(2)_{0,4}\subset \so(4)_{3,4}$ or to $\so(3)_{3,3}$. 

\begin{proposition}\label{Proposition.3.2.1} Suppose that  $\hat G/H$ admits a $\hat G$-invariant $G_2$-structure  such that $\bar\rho(\h)$ is one of possibilities 1-5 listed above, except case (2d).
Suppose that $\hat G = G ^{sc} \times T ^k$ where $G^{sc}$ is a simply connected semisimple Lie group and $\bar \rho$ is  a faithful representation.
Then  $(H\subset \hat G)$ must be one of the  pairs listed in  Propositions \ref{Proposition.2.2.1},  \ref{Proposition.2.2.3}, \ref{Proposition.2.2.5}, \ref{Proposition.2.2.6}, \ref{Proposition.2.2.7}, \ref{Proposition.2.2.8}. Here we make the same  assumptions for the cases   considered in  Propositions \ref{Proposition.2.2.6}, \ref{Proposition.2.2.7}  as in the previous section.
\end{proposition}

\begin{proof} 
Let $H^0$ be the identity connected component of $H$. Since $H^0 \subset SO(4)_{3,4}$  
the space $\Lambda ^3 (\R^3)$  is  invariant under the action of $H^0$ on $\Lambda ^3 (\R ^ 7)$.  Thus the invariance principle implies that  $\rho ( H)  \subset G_2$, if
and only if it belongs to $SO(4)_{3,4} = G_2 \cap (SO (3) \times SO(4))$.
This 
completes the  proof of Proposition \ref{Proposition.3.2.1}. 
\end{proof}

Let us consider the remaining cases. To handle the possibility (2d),  we use  our analysis in subsection 2.2. Case (2d) corresponds to  case  (2ii)
with the  associated  embedding of $\so(3) \to \so(5)$   being a real irreducible representation  of $\so(3)$ of dimension 5. Its connected
subgroup in $Spin (5) = Sp (2)$  is the subgroup $SU(2)_4$,  defined by the irreducible complex representation of $SU(2)$ of dimension  4.

In possibility 6, the argument  in  the previous section yields that
 there  is  no new  case.
 
In  possibility 7,  taking into account that $\dim \g = 21$, we conclude that $\g = \so (7)$.

In possibility 8, taking into account that $\dim \g = 15$,  we conclude that $\g = \su(4)$ or $\g =\g_2 + \R$.

\begin{proposition}\label{Proposition.3.2.2}  Let $\hat G = G ^{sc} \times T ^k$, where $G^{sc}$ is   a connected simply-connected semisimple Lie group. Suppose that  $\hat G/H$ admits a $\hat G$-invariant $G_2$-structure  such that  the action
of $\hat G$ is almost effective.  Suppose that $(H\subset   \hat G)$ is not  listed in Proposition \ref{Proposition.3.2.1}.
Then  $(H\subset \hat G)$ is one of the  pairs listed  below:\\
Possibility (2d), $H = SU(2)_4 \subset Sp (2)$.  The kernel of the action is $\Zz (\hat G) \cap H = \Z_2$.\\
Possibility (7), $H =  G_2 \cdot \Gamma  \subset Spin (7)$, $\Gamma \subset \Zz(Spin (7)) = \Z_2$. The  kernel of the  action is  $\Gamma$.\\
Possibility (8),  $SU (3)\cdot \Gamma \subset SU (4) $, $\Gamma \subset \Zz (SU(4)) = \Z_4$.  The kernel of the action is $\Gamma$.\\
Possibility (8), $SU(3) \cdot \Gamma \subset G_2 \times S^1$,  $\Gamma \subset \Zz (G_2 \times S^1)$. The kernel of the action is $\Gamma$.

\end{proposition}

\begin{proof} It suffices to consider  the case of disconnected Lie subgroups $H$. We have examined cases 1-6, except  (2d).
Applying  Schur's Lemma  to  possibilities (2d) and (7), we conclude that
$\Nn_{Sp (2)} (SU(2)_4)$ is $SU(2)_4$ and $\Nn_{Spin (7)} (G_2) = G_2 \cdot \Z_2$.
Applying the invariance principle  to possibility 8, we  conclude that
$\Nn_{SU(4)} (SU (3)) = S(U(3) \times U(1))$  and $\Nn_{G_2 \times S^1} (SU(3)) = \Z_2 [D_7] \times S^1$. Using the decomposition $S(U(3) \times U(1))= SU(3) \cdot \Zz_{SU(4)}(SU(3))$, we conclude  that the isotropy action of an  element
$g \in \Nn_{SU(4)} (SU (3)) = S(U(3) \times U(1))$   belongs to $G_2$, iff $g = g_1 \cdot h$ where $g_1 \in SU(3)$ and $h\in \Zz(SU(4)) = \Z_4$.  Finally  we check easily  that $Ad_{D^7}$ does not belongs to $G_2$, since $(S^6/\Z_2)\times S^1$
is not orientable. This proves Proposition \ref{Proposition.3.2.2}.
\end{proof}

\subsection{Classification theorem}

We summarize our  examination in the previous subsection in the following

\begin{theorem}
\label{Theorem.3.3.1}   Let $G/H$ be a homogeneous space admitting
a $G$-invariant $G_2$-structure.  We assume that $G$ is a connected compact Lie group  and $G$ acts effectively on $G/H$. Then $G/H$ is  one in  Theorem \ref{Theorem.2.3.1} or  one in the following list
\end{theorem}

\begin{math}
\begin{array}{ccc}
Case & G & H  \\%d_3(G/H) = d_1 + d_2 \\
(2d) & SO(5) & SO(3)_5 \\%&  d _1 = 0, d_2 = 1 ,\\
7 & Spin (7) &   G_2  \\% &  d_1 = 0, d_2 = 1,\\
7 & SO (7)  & G_2  \\%&, d_1 = 0, d_2 = 1,\\
8 & SU(4)  & SU(3)\\%  &  d_1 = 1, d_2  = 2.
8 & SU(4)/\Z_2 &  SU(3)\\
8 &  PSU(4) &  SU(3)\\
8 & G_2 \times S^1 &  SU(3)
\end{array}
\end{math}

\begin{remark}\label{Remark.3.3.2}  a) We  have the same formula  $d_3 = d_1 + d_2$ as in   the  case of $\tilde G_2$.
The dimension of  the space of all $G$-invariant $G_2$-structures on $G/H$ is $d_3(G/H) -1$.\\
b) Many  spaces among those listed in  Theorem \ref{Theorem.3.3.1} have been known before.
Case (2d) has been   treated by Bryant in \cite{Bryant1987} and Bryant and Salamon in \cite{BS1989}. Case (5ii) has been examined by  Cabrera, Monar and Swann
\cite{CMS1994}.  In \cite{FKMS1997} Friedrich and his coauthors classified all  simply-connected  compact homogeneous  nearly parallel $G_2$-manifolds.   
We remark that a large part of homogeneous spaces listed in Theorem \ref{Theorem.3.3.1} are quotients of spaces listed
in \cite{FKMS1997}.
\end{remark}

\section {Spaces $G/H$ with  high rigidity  or with low rigidity}

In this section we  consider several  examples  of  spaces $G/H$ with high rigidity or  low rigidity.  Many of these examples
are known, but we provide   simpler proofs   of some known results  based on our classifications. We also  present
some new results.

\subsection{Spaces $G/H$ with $d_3 (G/H) = 1$}
Let $G/H$ be one of homogeneous spaces listed in Theorem \ref{Theorem.2.3.1} or Theorem \ref{Theorem.3.3.1}. Clearly $d_3(G/H) = d_1 + d_2$   is equal 1, if and only if  $d_1 = 0$ and $d_2 = 1$, so $G/H$ is in possibility (2d) or possibility  (7).  In other words invariant  positive forms $\phi$ on these spaces are defined  uniquely  up to rescaling.
These spaces are well studied  before \cite{Bryant1987}, \cite{FKMS1997}.  They are nearly parallel $G_2$-manifolds, i.e.
\begin{equation}
d\phi = \lambda * \phi
\label{4.1.1}
\end{equation}
for some constant $\lambda \not = 0$. 
We will give a brief explanation of this fact, which is close to the argument in \cite{Bryant1987}.  It is easy to see that equation (\ref{4.1.1}) holds, because $d_3(G/H ) =1$. 
To prove $\lambda \not = 0$, we observe that $d* \phi = 0$, since   there is no $\rho(H)$-invariant
2-form on $V$. 
On the other hand, by \cite{AK1975} 
there is no invariant metric with zero Ricci curvature  on $G/H$ \cite{AK1975}. Hence $\lambda \not = 0$.

\subsection{Spaces  $G/H$ with $d_3(G/H ) = 2$}  
They are the spaces in possibilities (1),  (2ci) with a nontrivial $\Gamma$,
and in possibility (3aiii) listed in  subsection 3.1.  These spaces present an interesting  class, since
there are a 1-parameter family of inequivalent  $G$-invariant  $\tilde G_2$-structures on $G/H$,  and a one-parameter
family of inequivalent    $G$-invariant  $G_2$-structures on $G/H$.

 An example  of $ \Gamma$  in the possibility (2ci) 
is the icosahedral rotation group of order 60 which is isomorphic to  the alternating group
$A_5$. The space is a quotient of a sphere $S^7$ by  $\Gamma$.

\begin{lemma}\label{Lemma.4.2.4}   The dimension of   the space of invariant 2-forms on $G/H$  with $d_3 = 2$ is less than or equal to 1. Any  $G$-invariant 2-form on $G/H$ is closed.
\end{lemma}

\begin{proof} The condition $d_3 = 2$ implies that $d_1 = 0$, since $d_2 \ge  d_1 + 1$.
Now we use the following decomposition of $G_2$-modules, see e.g. \cite{Bryant2005}
$$ \Lambda ^ 2(\R ^ 7 ) ^ * =  \g_2 +  \R ^ 7 .$$
Since $\rho(H) \subset G_2$, the above decomposition is  invariant under the $\rho(H)$-action.
Since $d_1 = 0$,  the existence of  a $G$-invariant 2-form on $G/H$ is equivalent to the  existence of  a non-trivial
centralizer $c$ of $\rho (H)$ in $\g_2$. Thus   either $\h = \so (3) $, or $\h = \so (3) + \R$.
In the first case, using our classification,  we conclude that it is  case (2cii) with $\h = \su (2)_{0,4}$. Considering  the decomposition  of $\Lambda ^2 (\R^7) ^*= \Lambda ^2 (\R ^3 \oplus \R ^4)^*$ with respect to the representation of $\h =\su (2)_{0,4}$,  we conclude that there exists a vector in $\Lambda ^2 (\R ^3)^* = \R ^3\subset \R^7$, which
is invariant under the action of $\rho(H)$. This contradicts our  remark above that $d_1 = 0$.
In the second case, since $rk\, \h = 2$, we conclude that $ c$ lies in the component $\R \subset \h$. 
In fact it is case (3aiii).  This proves the first assertion of Lemma \ref{Lemma.4.2.4}, and it gives
rise to a unique (up to rescaling)  $G$-invariant  2-form $\om$ on $G/H$  as follows. We write $H = H_0 \cdot U(1)$.  Let us consider the $U(1)$-fibration 
$G/H_0 \to G/ (H_0 \cdot U(1))$ whose  fiber is $U(1)/ ( U(1) \cap H^0))$.  The form $\om$ is  the curvature of   this
non-trivial $U(1)$-fibration. 
Thus $\om$ is a representative of  a   $G$-invariant 2-form which is unique up to rescaling. Since it is closed, Lemma \ref{Lemma.4.2.4} follows  directly.
\end{proof}

\begin{theorem}\label{Theorem.4.2.5}  Let $G/H$ be a  compact homogeneous manifold 
with $d_3(G/H) =  2$. \\
a) Any $G$-invariant $\tilde G_2$-structure and any $G$-invariant $G_2$-structure on $G/H$ is coclosed.\\
b) There exists  a unique $G$-invariant nearly parallel $G_2$-structure on $G/H$.
\end{theorem}

\begin{proof}
a)  It suffices to prove that  $d\psi ^4= 0$,  for any $G$-invariant stable 4-form $\psi$.  We will show that
the pairing of $*_\psi d\psi ^4 $ with any $G$-invariant 2-form $\om$ is zero.   This pairing is equal   to  the pairing
of $\psi$ and $d\om$. By Lemma \ref{Lemma.4.2.4} this pairing is zero.

b) The existence and uniqueness of a $G$-invariant nearly parallel $G_2$-structure on  these spaces follows from  
 a   computation  of the rank of a 4-form  $d\phi$, here $\phi$ is a $G$-invariant  3-form   on $G/H$, combining with the  following observation.
 For all $G$-invariant 3-forms $\phi$, all  4-forms $d\phi$  are in the same conformal class.
To  prove this, we  use the  assertion that the dimension $d^1$ of the space  of all 4-forms $d\phi$  is  equal to 1, where $\phi$ is a G-invariant 3-form  on $G/H$.  To see it, we note that $d^1$ is  less than or equal to
2. On the other hand, since the restriction of the Cartan form $\Om ^3$  to $V$:
$$\Om ^ 3(X, Y, Z) = \langle X, [ Y, Z] \rangle$$
is not zero on our  spaces $G/H$  with $d_3  = 2$, and using $d\Om^3 = 0$, we conclude that $d^1 \le 1$.
A simple  computation shows that  $d^1 \not = 0$. Hence $d^ 1 = 1$.
Consequently, for all $G$-invariant 3-forms $\phi$, all  4-forms $d\phi$  are in the same conformal class. 
\end{proof}

\begin{remark}\label{Remark.4.2.6} The existence of   $G$-invariant nearly parallel $G_2$-structures on    spaces $G/H$ with $d_ 3 = 2$
and $\pi_1 (G/H) = 0$ has been  established in \cite{FKMS1997} by a different method. In \cite{Hitchin2001}  Hitchin  suggests  a variational method to find  nearly parallel $G_2$-structures.
\end{remark}

\begin{example} \label{Example.4.2.9} We   consider case (1), see a detailed description in subsection 3.1.
Using a  method in \cite{Bryant1987} and \cite{Hitchin2000} we   explain how to find  all
$G$-invariant  $\tilde G_2$-forms and  $G$-invariant $G_2$-forms  on $G/H$. Recall that $V = W + W^\perp$.  Take an orthogonal basis $(e_1, e_2, e_3)$   in $W ^\perp$.
 We  choose another quaternion  basis $e_4, e _5, e _6, e _7 \in W $   with respect to the action of $\h$.  Let
 $e^i$  be  the dual basis  in $V ^*$.
 Then the 3-form $\om^{123}$ and   the 3-form $\phi$ defined in Definition \ref{Definition.3.1.1} are generators
 of our space of $\rho (H)$-invariant  3-forms on $V$. 
 The space of  $\rho(H)$-invariant 4-forms on $V$ is generated by
 $$\psi_1 = \om ^{4567}, \, \psi _2 =  \om ^{4567} + \om^{2367} + \om ^{2345} + \om ^{1357}- \om^{1346} - \om ^{2356} -\om^{1247}.$$
 Any $\rho(H)$-invariant 4-form $\psi(a,b)$ on $V$   is of the form
$a\psi_1 +  b \psi_2$.   We define  the associated 2-bilinear  form $g_{\psi(a,b)}$ on $ V ^* \otimes V^* $  with value in
$[\Lambda ^7 (V ^*)]^2$ by setting \cite[8.4]{Hitchin2001}
$$g_{\psi(a,b)}(X^*, Y^*) = (X^* \wedge \psi(a,b))\wedge (Y^* \wedge \psi (a,b))\wedge \psi(a,b).$$
Since  this is an invariant metric, and $d _2 = 2$, we calculate easily
$$ g_{\psi (a,b)} = (a^2 (2a+ 3b) [( e_1) ^2 + (e_2)^2 + ( e _3) ^2] + 3a^3[ (e_4 ) ^2 + ( e_5) ^2 + ( e _6) ^2 +  (e _7)) ^2](\om^{1234567})^2.$$
Hence  $vol\, (\psi(a,b)) = (a) ^{3/2} (2a + 3b) ^{1/4} (3) ^{1/3}\om ^{1234567}$. Thus $\psi (a,b)$ is a stable 4-form, if and only if
$a(2a +3b) \not = 0$.  If $(2a +3b)a > 0$ then  $*_{\phi(a,b)}\phi(a,b)$ is a $G_2$-form, if  $(2a+3b) a < 0$ then
$*_{\phi (a,b)}\phi(a,b)$ is a $\tilde G_2$-form.  
\end{example}

\subsection{Spaces with $d_3 (G/H) =  35$}  It is easy to see that $d_3 (G/H) \le  35$, and   the equality is
attained, if and only if $H$ is trivial.  On $G = T^7$  any  $G$-invariant  $G_2$-structure (or $\tilde G_2$-structure)
is torsion-free.  Now let us look at the  next non-trivial case  with $ G = SU(2) \times T^ 4$  or
$G= SO(3) \times T^4$.

\begin{proposition} \label{extra} Let $G$ be $SU(2) \times T^4$ or $SO(3) \times T^4$.\\
(i)  There is no $G$-invariant nearly  $G_2$-structure  on $G$.\\
(ii) There is no  $G$-invariant closed  stable 3-form  on $G$.\\
(ii) The dimension of the space of coclosed $\tilde G_2$-forms  as well as the dimension  of the space of coclosed $G_2$-forms on $G$
is 19.
\end{proposition}

\begin{proof} (i) The existence of a nearly  $G_2$-structure implies that the associated metric is 
Einstein.  By a theorem due to Alexeevskii and Kimeldeld \cite{AK1975},  the  Ricci curvature of   the associated  metric is positive, which implies that  the  fundamental group of $G$ is finite. So we obtain a contradiction.

(ii) Let us choose $e_1,e_2,e_3 \in  \su (2)$  such that
$[e_1,e_2] = e_2, [e_1,e_3] = -e_3 $ and $[e_2,e_3] = e_1$. Let  $e_4,e_5,e_6, e_7 \in  lT^4$, the Lie algebra of $T^4$. Suppose that  there exists a closed stable 3-form $\phi$.  Assume that $\phi = c\om_{123} + \phi_0$, where
$\phi_0  ( e^{123}) = 0$. Then $\phi _0 = d\psi ^2 +  V_1\rfloor \om ^{4567}$, where $\psi ^2$  is a $G$-invariant 2-form on $G$  and
$V_1\in lT^4$.
Since $d (e^1 \wedge e^2 ) = 0 = d( e^2 \wedge  e^3) = d ( e^3 \wedge e^1)$, we can assume that $\psi ^2 $ has the form $ a_1 e^1 \wedge f^1 + a_2 e^2 \wedge f^2 + a_3 e ^3 \wedge f^3$ where $f^i \in  (lT^4) ^*$.
Thus  $\phi= c\om ^{123} + a_1 e^2 \wedge e^3 \wedge f^1 + a_2 e^1 \wedge e^2\wedge f^2 + a_3 e^3 \wedge e^1 \wedge f^3 + V_1\rfloor \om ^{4567}$. Let $V_2 \in lT^4\setminus \{0\}$ such that $f^i (V_2) = 0$ for $i = \overline{1,3}$. If $V_1 = bV_2$,   then $rk (\phi) \le 6$, so $\phi$ is not stable.  If $V_1$ and $V_2$ are linearly independent, let us consider  a space $\R^6 \subset \g$ containing $(e_1,e_2, e_3, V_1)$ which is   a complement to $V_2$. Let $\theta$ be the 1-form on $\g$ such that $\theta (V_2) = 1$ and $\theta _{|\R^6} = 0$.
Then  $\phi =   \theta \wedge \gamma_1 + \gamma _2$, where $\gamma _1\in \Lambda ^2(\R^6)^* $ has length 1  and $\gamma_2 \in \Lambda ^3 (\R^6) ^*$.  By \cite[Lemma 2]{BV2003}, which reformulates  results in \cite{Westwick1981}, $\phi$  is not stable.

(iii)  We set
$$\phi _{\pm} = \pm \om^{123} + \om ^{145} +\om^{167} + \om ^{246} - \om^{257} - \om^{347} -\om^{356}.$$
A simple calculation shows that $*_\phi\phi_{\pm} =  \pm \om ^{4567} +d\psi ^3$, where $\psi ^3  = d( \om^{167} +
\om^{145} -\om^{357} + \om ^{346} - \om ^{256} -\om^{247})$, hence $*_\phi\phi_{\pm}$ is a closed form.
Clearly the dimension of  the family  of  coclosed stable forms at $\phi_{\pm}$  is equal  to  $\dim ( d  \Om_G ^ 3 ) +5 (=b_4(G))$.
Now we compute
$$\dim (d ( \Om ^ 3 _ G ) ) = 35 -\dim \ker d_{| \Om ^ 3_G} , $$
$$\dim (\ker  d _{| \Om ^ 3 _G } ) =  5 + \dim ( d ( \Om ^2 _G)), $$
$$\dim ( d ( \Om ^ 2_G)) =  21 - \dim \ker d _{| \Om ^ 2 _G}, $$
$$\dim \ker ( d _{|\Om ^ 2 _G }) = 2 + \dim  ( d ( \Om ^1 _G)) =  5.$$
Thus the dimension of the space of stable 3-forms at $\phi_{\pm}$ is  19.  Applying this argument to other
stable  invariant  3-forms  we complete the proof of Proposition \ref{extra}.
\end{proof}

\begin{remark}
All the spaces considered above admit stable closed 4-forms. Using a method in \cite{Hitchin2001}   we can construct metrics with $Spin (7)$-holonomy or metric
with $Spin (4,3)$-holonomy on  the product  of these spaces  with an interval.  Hitchin considered only  $Spin(7)$-holonomy, but     his arguments are applied to the case of $\tilde G_2$-structure and $Spin (3,4)$-holonomy. 
\end{remark}

{\bf Acknowledgement.} 
We thank Sasha Elashivili, Ines Kath and  Ji\v ri Van\v zura for helpful remarks. 
We thank the anonymous referees  for   helpful  suggestions  and corrections, which  improve our note greatly. H.V.L. is supported in part by Grant of ASCR IAA100190701.
A part of this  paper has been done while H.V.L. was visiting the ASSMS, GCU Lahore-Pakistan,  and the University of Toulouse. She thanks  these institutions  for their hospitality
and financial support.  M.M. is supported in part by HEC of Pakistan.

{\it Note added in proof}. We thank Ines Kath for showing us  a paper by F. Reidegeld \cite{Reidegeld2010}, which partially overlaps with our results.

\bigskip

{\small  Address: Mathematical Institute of ASCR,
Zitna 25, CZ-11567 Praha 1, email: hvle@math.cas.cz\\
and  ASSMS, Government College University, Lahore-Pakistan, email:mobeenmunir@gmail.com}

\end{document}